\documentclass[11pt,a4paper,twoside]{amsart}

\usepackage{amsmath,amsfonts,amsthm,amsopn,color,amssymb,enumitem,cite,soul}
\usepackage{palatino}
\usepackage{graphicx}
\usepackage[normalem]{ulem}
\usepackage[colorlinks=true,urlcolor=blue,citecolor=red,linkcolor=blue,linktocpage,pdfpagelabels,bookmarksnumbered,bookmarksopen]{hyperref}
\hypersetup{urlcolor=blue, citecolor=red, linkcolor=blue}
\usepackage[left=3.4cm,right=3.4cm,top=2.72cm,bottom=2.72cm]{geometry}

\usepackage[hyperpageref]{backref}

\usepackage[colorinlistoftodos]{todonotes}
\makeatletter
\providecommand\@dotsep{5}
\def\listtodoname{List of Todos}
\def\listoftodos{\@starttoc{tdo}\listtodoname}
\makeatother

\newcommand{\e}{\varepsilon}
\newcommand{\eps}{\varepsilon}
\newcommand{\C}{\mathbb{C}}
\newcommand{\R}{\mathbb{R}}

\newcommand{\RN}{{\mathbb{R}^N}}
\newcommand{\RD}{{\mathbb{R}^2}}

\newcommand{\de}{\partial}
\newcommand{\weakto}{\rightharpoonup}

\renewcommand{\le}{\leslant}
\renewcommand{\ge}{\geslant}
\renewcommand{\a }{\alpha }

\renewcommand{\d }{\delta }
\newcommand{\vfi}{\varphi}
\newcommand{\g }{\gamma }

\newcommand{\n }{\nabla }
\newcommand{\s }{\sigma }

\renewcommand{\O}{\Omega}

\renewcommand{\S}{\Sigma}

\newcommand{\calh}{\mathcal{H}}
\newcommand{\ch}{\mathcal{H}^{2,4}(\RD)}
\newcommand{\chr}{\mathcal{H}_r^{2,4}(\RD)}
\renewcommand{\H}{H^1(\RD)}
\newcommand{\Hr}{H^1_r(\RD)}

\newcommand{\M}{\mathcal{M}}
\renewcommand{\S}{\mathcal{S}}
\newcommand{\N}{\mathbb{N}}

\renewcommand{\C}{\mathbb{C}}
\renewcommand{\o}{\omega}

\newcommand{\inb }{\int_{B_R}}
\newcommand{\ird }{\int_{\RD}}

\def\bbm[#1]{\mbox{\boldmath $#1$}}
\newcommand{\beq }{\begin{equation}}
\newcommand{\eeq }{\end{equation}}

\renewcommand{\le}{\leqslant}
\renewcommand{\ge}{\geqslant}
\newcommand{\dis}{\displaystyle}

\newtheorem{theorem}{Theorem}[section]
\newtheorem{lemma}[theorem]{Lemma}

\newtheorem{definition}[theorem]{Definition}
\newtheorem{proposition}[theorem]{Proposition}
\newtheorem{remark}[theorem]{Remark}
\newtheorem{corollary}[theorem]{Corollary}

\title[Positive energy static solutions for the Chern-Simons-Schr\"odinger system]{Positive energy static solutions for the Chern-Simons-Schr\"odinger system under a large-distance fall-off requirement on the gauge potentials}

\author[A. Azzollini]{Antonio Azzollini}
\address{Dipartimento di Matematica, Informatica ed Economia, Universit\`a degli
	Studi della Basilicata,
	\newline\indent
	Via dell'Ateneo Lucano 10, I-85100
	Potenza, Italy}
\email{antonio.azzollini@unibas.it}

\author[A. Pomponio]{Alessio Pomponio}
\address{Dipartimento di Meccanica, Matematica e Management,
	Politecnico di Bari
	\newline\indent
	Via Orabona 4,  70125  Bari, Italy}
\email{alessio.pomponio@poliba.it}

\thanks{The authors are supported by PRIN 2017JPCAPN {\em Qualitative and quantitative aspects of nonlinear PDEs}.}

\begin{document}
	
	\maketitle

	\begin{abstract}
	In this paper we prove  the existence of a positive energy static solution for the Chern-Simons-Schr\"odinger system under a large-distance fall-off requirement on the gauge potentials. We are also interested in existence of ground state solutions. 
	\end{abstract}

	\section{Introduction}

	The following Chern-Simons-Schr\"odinger system 
	\beq \label{eq:e0}\tag{$\mathcal{CSS}$}
	\begin{array}{l}
		i D_0\phi + (D_1D_1 +D_2D_2)\phi + |\phi|^{p-1} \phi =0,\\
		\partial_0 A_1  -  \partial_1 A_0  = {\rm Im}( \bar{\phi}D_2\phi), \\
		\partial_0 A_2  - \partial_2 A_0 = -{\rm Im}( \bar{\phi}D_1\phi), \\
		\partial_1 A_2  -  \partial_2 A_1 =  \frac 1 2 |\phi|^2,
\end{array}
	\eeq
	has been object of interest for many authors, physicists and mathematicians, in the last thirty years. 
	
	For $p=3$, it corresponds to the model proposed by Jackiw-Pi \cite{jackiw0}, and studied also in \cite{hagen, hagen2, jackiw, jackiw1, jackiw2},  to describe the dynamics of a nonrelativistic solitary wave that behaves like a particle, in the three dimensional gauge Chern-Simons theory. 
	
	Here $t \in \R$, $x=(x_1, x_2) \in \R^2$, $\phi :
	\R \times \R^2 \to \C$ is the scalar field, $A_\mu : \R\times \R^2
	\to \R$ are the components of the gauge potential and $D_\mu =
	\partial_\mu + i A_\mu$ is the covariant derivative ($\mu = 0,\
	1,\ 2$).
	%This system describes nonrelativistic matter interacting with Chern-Simons gauge 
	%fields in the plane,  \cite{hagen, hagen2, jackiw0, jackiw, jackiw2}.
	%This model was first proposed and studied in \cite{jackiw0,
	%jackiw, jackiw2}, and is known also with the name of
	%Chern-Simons-Schr\"{o}dinger equation. 
	\\
	The initial value problem, wellposedness, global existence and blow-up, scattering, etc. have been considered in
	\cite{berge, huh, huh3, liu, tataru, oh} for the case $p=3$. In particular Jackiw and Pi were able to find self-dual solitons deduced by {\it static solutions} of \eqref{eq:e0}  transfomed by means of Galilean boost or conformal invariance.
	%See also \cite{liu} for a global existence result in the defocusing case.
	\\
	Since, as usual in Chern-Simons theory, problem \eqref{eq:e0} is
	invariant under the gauge transformation
	\beq \label{gauge} \phi \to \phi e^{i\chi}, \quad A_\mu \to A_\mu
	- \partial_{\mu} \chi \eeq 
	for any arbitrary $C^\infty$ function $\chi:\R\times\RD\to\R$, we easily see that the definition of static solution, that is time-indipendent solution, makes sense once we have removed the gauge freedom. In \cite{jackiw0} it has be done assuming the Coulomb gauge choice $\n \cdot\bf A=0$ (here ${\bf A}=(A_1,A_2)$),
	supplemented by large-distance fall-off requirements on the differential equations
	satisfied by $A_0, A_1$ and $A_2$ (see \cite{jackiw2}). In particular, we require that
	\begin{equation}\label{eq:a012}\tag{$\mathcal{FO}$}
	A_0(x)= O(1/|x|),\;
	|{\bf A}(x)|= O(1/|x|),
	\end{equation}
	being this asymptotic behaviour physically relevant, as it is the reflection of the possible presence of, respectively, electric charges and magnetic monopoles.

	The existence of standing waves for \eqref{eq:e0} and general
	$p>1$  has been studied in \cite{BHS,cunha, huh2, AD,AD2,WT,Y}, whereas standing waves with a vortex point have been studied in  \cite{BHS2,JPR} (see also the review paper \cite{P}).
	
	In order to find standing waves, we introduce the following ansatz
	\begin{equation}\label{ansalz}
	\begin{array} {lll} \phi(t,x) = u(|x|) e^{i \omega t}, && A_0(t,x)= A_0(|x|),
	\\
	A_1(t,x)=\dis -\frac{x_2}{|x|^2}h(|x|), && A_2(t,x)= \dis \frac{x_1}{|x|^2}h(|x|),
	\end{array}
	\end{equation}
	where $\o\in\R$ is a given frequency and  $u$ is a radial real valued function that, with an abuse of notation, has to be meant as a one or two variables function according to the situation.\\
In \cite{BHS} the authors proved that $(\phi, A_0,A_1,A_2)$ solves \eqref{eq:e0} if we set 
		$$h(r)=h_u(r)= \frac 12\int_0^r s u^2(s)  \, ds, \qquad r>0,$$
		in the previous ansatz \eqref{ansalz},
		$$A_0(x)= \xi + \int_{|x|}^{+\infty} \frac{h_u(s)}{s} u^2(s)\, ds,$$
		with $\xi \in \R$ arbitrary, and
		$u$ is a solution of the equation
		\beq\label{equation}  
		- \Delta u + \left( \o + \xi+ \dis
		\frac{h_u^2(|x|)}{|x|^2} +  \int_{|x|}^{+\infty} \frac{h_u(s)}{s}
		u^2(s)\, ds   \right)  u  =|u|^{p-1}u, \quad \hbox{ in }
		\RD.
		\eeq
	Therefore, given a standing wave solution
	\begin{equation*}
	\left(u(x)e^{i\o t},\xi + \int_{|x|}^{+\infty} \frac{h_u(s)}{s} u^2(s)\, ds, -\frac{x_2}{|x|^2}h(|x|), \frac{x_1}{|x|^2}h(|x|)\right),
	\end{equation*} 
	we can consider, for any $c\in\R$, the function $\chi(t) = c\, t$ and use the
	gauge invariance \eqref{gauge} to obtain the family of standing wave solutions 
	\begin{equation*}
	\left(u(x)e^{i(\o+c) t},\xi - c + \int_{|x|}^{+\infty} \frac{h_u(s)}{s} u^2(s)\, ds, -\frac{x_2}{|x|^2}h(|x|), \frac{x_1}{|x|^2}h(|x|)\right)_{c\in\R}
	\end{equation*}
	which is characterized by the constant $\omega + \xi$ that results to be a gauge invariant. \\
	In order to differentiate and classify the solutions, as in \cite{jackiw2} we fix the gauge freedom imposing the following  {\it decay at infinity condition} on the potential $A_0$   
	\begin{equation}\label{eq:a0}
	\lim_{|x|\to +\infty}A_0(x)=0.
	\end{equation}
	
	We point out that, assuming the square integrability of $u$ (which, as we are going to show, means that the solution has a finite total charge), our ansatz, together with \eqref{eq:a0}, is consistent with the Coulomb gauge choice $\n \cdot {\bf A} =0$,
	supplemented by large-distance fall-off requirements \eqref{eq:a012}.
	
	According to the above discussion, in what follows we will
	take $\xi=0$ which is a necessary condition for \eqref{eq:a012} as it is assumed for example in \cite{berge, jackiw2}.
	\\
	Equation \eqref{equation}, therefore,  becomes
	\beq\label{corretta}  
	- \Delta u + \left( \o +  \dis
	\frac{h_u^2(|x|)}{|x|^2} +  \int_{|x|}^{+\infty} \frac{h_u(s)}{s}
	u^2(s)\, ds   \right)  u  =|u|^{p-1}u, \quad \hbox{ in }
	\RD,
	\eeq
	Observe that static solutions of \eqref{eq:e0} having the form \eqref{ansalz} are deduced from \eqref{corretta} for $\o=0$. \\
	
	Static solutions of \eqref{eq:e0} deduced from \eqref{corretta} have been found only when $p=3$  in \cite{BHS}. In detail, in \cite{BHS} the authors proved that when $p=3$ solutions to \eqref{eq:e0} satisfying the ansatz \eqref{ansalz} and which have a field of matter that is nowhere zero (in the sense that $u>0$ everywhere) must be static and belong to a one-parameter family which can be explicitly described. In particular, it is quite interesting to observe that such solutions are real valued, differently from the complex valued static field of matter found in \cite{jackiw0}. 
	Both solutions found in \cite{BHS} and those found in \cite{jackiw0}  have zero energy (see \cite[sec.5]{BHS} and \cite[sec.4]{jackiw2}). 	%They come as solutions of a self-dual equation, which leads to a Liouville equation that can be completely integrated.
	
	When $p>1$, $p\neq 3$, equation \eqref{corretta} has been approached by variational methods looking for non-static solutions of \eqref{eq:e0} with $\o>0$. Indeed as showed in
	\cite{BHS}, the equation \eqref{corretta} is nonlocal and it corresponds
	to the Euler-Lagrange equation of the functional $
	I_{\o}: \Hr \to \R$,
	\begin{equation}\label{functional}
	I_{\o}(u)  = \frac 12  \|\nabla u\|_2^2 + \frac \o 2 \|u\|_2^2 
	+\frac{1}{2} \int_{\R^2}\frac{h_u^2 u^2}{|x|^2} dx  - 
	\frac{1}{p+1}  \|u\|^{p+1}_{p+1},
	\end{equation}
	where $$\Hr:=\{u\in \H:u \hbox{ is radially symmetric} \}.$$
	
	Observe that $I_{\o}$ presents a
	competition between the nonlocal term and the local nonlinearity
	of power-type.
	
	When $p>3$, in \cite{BHS} the authors showed that $I_\o$  is unbounded from below and exhibits a mountain-pass geometry.  However the existence of non-static solutions is not so direct, since for $p \in (3,5)$ the Palais-Smale condition is not known to hold. This problem is bypassed by using a constrained minimization taking into account the Nehari and Pohozaev identities. Up to our knowledge, there is no information about the sign of the energy of these solutions. 
	
	Finally, non-static solutions of \eqref{eq:e0} deduced from \eqref{corretta} are found for $p\in(1,3)$ in \cite{BHS} as minimizers on a $L^2$-sphere: here the gauge freedom is exploited to combine the value $\o$ with a Lagrange multiplier, generating a family of non-static, not gauge equivalent solutions which do not in general satify the large-distance falling-off condition. \\
	Later, the result for $p \in  (1,3)$ has been extended in \cite{AD} 
	by investigating the geometry of $I_\o$. 
	Through a careful analysis for a limit equation, the authors showed that there exist $0 < \o_0 < \tilde \o < \bar \o$ such that if $\o> \bar \o$, 
	the unique solutions to \eqref{corretta} are the trivial ones; if $\o_0 < \o < \tilde \o$, there are at least two positive solutions to 
	\eqref{corretta}; if $0 < \o < \o_0$, there is a positive solution to \eqref{corretta} for almost every $\o$. 
	\\
	In particular, in \cite{AD} the authors proved that one of the two solutions found in the interval $(\o_0, \tilde \o)$ has negative energy.\\
	We mention, moreover,  \cite{cunha,huh2} where multiplicity results are provided.\vspace{0.2cm}
	
	Inspired by the original paper by Jackiw and Pi \cite{jackiw0} and the following literature, the aim of this paper is to study \eqref{eq:e0} looking for positive energy solutions. \\
	We recall the following result that can be easily deduced by the definition of energy and charge and direct computations
	\begin{proposition}
		Assume that $(\phi, A_0, A_1, A_2)$ is a solution of \eqref{eq:e0} satifying the ansatz \eqref{ansalz}. Then the energy and the charge of the solution are, respectively, 
		\begin{align}
		E (u)&=\frac 12  \|\nabla u\|_2^2 
		+\frac{1}{2} \int_{\R^2}\frac{h_u^2 u^2}{|x|^2} dx  - 
		\frac{1}{p+1}  \|u\|^{p+1}_{p+1},\label{energy}\\
		Q(u)&=\frac 12 \| u\|_2^2.\nonumber%\label{charge}
		\end{align}
	\end{proposition}
	By a comparison between \eqref{functional} and \eqref{energy}, we see that $E=I_0$, that is \eqref{corretta} corresponds to the Euler-Lagrange equation of the functional of the energy, when we are looking for static solutions.
	
	%From a mathematical point of view, the case $\o=0$, at contrary, is much less studied even if it is very interesting both from a physical point of view and a mathematical point of view. Observe, indeed, that any standing wave $(\phi,A_0,A_1,A_2)$ of the type \eqref{ansalz} with $\o=0$ is a static solution of \eqref{eq:e0}. 
	%When $p=3$, this kind of solutions  has been studied in \cite{jackiw0,jackiw2} where its explicit form is provided solving a Liouville type equation. 
	
	%  the problem is in the so called {\em positive mass case} and the functional $I_\o$ is well defined in $\Hr$ and standard variational tools work properly. 
	%Moreover, since each solution belongs to $L^2(\RD)$, the total charge (which, up to a constant factor, coincides with the $L^2$-norm) is finite. This implies, furthermore, that $A_1(x)$ and $A_2(x)$ vanish whenever $|x|\to +\infty$, as $A_0(x)$. 

	From a mathematical point of view, the equation
	\begin{equation}\label{eq:static}
	- \Delta u + \left(  \dis
	\frac{h_u^2(|x|)}{|x|^2} +  \int_{|x|}^{+\infty} \frac{h_u(s)}{s}
	u^2(s)\, ds   \right)  u  =|u|^{p-1}u, \quad \hbox{ in }
	\RD,
	\end{equation}
	falls in that class which is usually called {\em zero mass equations}. A variational approach to it immediately presents several difficulties, starting with the definition of a suitable functional setting.  Indeed, at least formally, solutions of \eqref{eq:static} can be found as critical points of the functional 
	$E$ for which, differently from the case $\o>0$, the space $H^1_r(\RD)$ seems to be ``too small'' to apply the techniques of the calculus of variations in a usual way. On the other hand, the idea of introducing the functional framework as a specific Sobolev space endowed with a norm containing an expression of the nonlocal term (see for example Ruiz' approach in \cite{R}) does not seem to be immediately applicable. In order to overcome this difficulty, we will make use of a perturbation argument as that presented inside \cite{BBS}, where the problem of defining the functional setting is due to the dimension $N=2$, and recovered in \cite{APP} where another type of nonlocal equation is considered in the zero mass case.

Combining equation \eqref{eq:static} with a condition at infinity, the problem reads as follows
\begin{equation}\label{eq}\tag{$\mathcal{P}$}
\begin{cases}
\dis -\Delta u 
+\left(\frac{h_{u}^{2}(|x|)}{|x|^2} 
+ \int_{|x|}^{+\infty}\frac {h_u(s)}{s}u^2(s)\,ds\right)u=
|u|^{p-1}u \qquad \hbox{in } \R^2, 
\\[5mm]
u(x)\to 0 \qquad \hbox{as }|x| \to +\infty,
\end{cases}
\end{equation}
where $u:\R^2 \to \R$ is radially symmetric and
$p>3$.

As a first step, we have to clarify what we mean as {\it solution of \eqref{eq}}. We start with the solutions in the {\em sense of distribution}. 

	\begin{definition}
			We say that a measurable function $u:\RD\to\R$ is a solution of \eqref{eq} in the {\em sense of distribution} if
				\begin{itemize}
					\item[1.] $u$ is in $L^p_{\rm loc}(\RD)$, \\
					\item[2.] for every $\vfi\in C_0^{\infty}(\RD)$
						$$
							\frac{u(x)\vfi(x)}{|x|^2}\left( \int_{B_{|x|}}u^2 dy\right)^2\in L^1(\RD)\hbox{ and }\frac{u^2}{|x|^2}\left( \int_{B_{|x|}}u^2 dy\right)\left( \int_{B_{|x|}}u\vfi\, dy\right)\in L^1(\RD),
						$$
					\item[3.] the operators
						\begin{align*}
							&\vfi\in C_0^{\infty}(\RD)\mapsto \ird\frac{u(x)\vfi(x)}{|x|^2}\left( \int_{B_{|x|}}u^2 dy\right)^2\,dx\\
							&\vfi\in C_0^{\infty}(\RD)\mapsto \ird\frac{u^2}{|x|^2}\left( \int_{B_{|x|}}u^2 dy\right)\left( \int_{B_{|x|}}u\vfi\, dy\right)\,dx
						\end{align*}
						are in $\mathcal D '$,
					\item[4.] for every $\vfi\in C_0^{\infty}(\RD)$
						\begin{multline*}
							\ird -u\Delta\vfi\, dx +  \ird\frac{u(x)\vfi(x)}{|x|^2}\left( \int_{B_{|x|}}u^2 dy\right)^2\,dx\\
							+ \ird\frac{u^2}{|x|^2}\left( \int_{B_{|x|}}u^2 dy\right)\left( \int_{B_{|x|}}u\vfi\, dy\right)\,dx =\ird |u|^{p-1}u\vfi\,dx,
						\end{multline*}
					\item[5.] for every $\delta >0$ the Lebesgue measure of the set $\{x\in\RD : u(x)\ge \d\}$ is finite.
				\end{itemize}
	\end{definition}

Even if solutions in the sense of distribution have of course mathematical relevance, it is absolutely clear that they are in general too weak for having any phisical significance. Indeed observe that, without any global integrability information, we are not able to prevent the infinite energy phenomenon arising, as it is well known, in classical electrodynamics models.

Then we introduce  a new setting
and proceed with the definition of solution in a stronger sense.

\begin{definition}
We define the  %functional 
sets $\ch$ and $\chr$ as the completion respectively of $C_0^\infty(\RD)$ and of the set of radial functions in $C_0^\infty(\RD)$ with respect to the norm $\|\cdot\|_{2,4}=\|\n \cdot\|_2+\|\cdot\|_4$.\\ Moreover, %we say $u\in \chr$ has {\em finite energy} if  $\frac{h_u(|x|)}{|x|}u\in L^2(\RD)$ and 
we denote by 
\begin{equation*}
\calh:=\{u\in \chr : E(u)  \hbox{ is finite} \}.
\end{equation*}
\end{definition}

We will discuss the properties of $\ch$ and $\chr$ in Section \ref{se:ff}.

\begin{definition}\label{def:ws}
Let $u\in \chr$. We say that $u$ is a {\em weak solution} of \eqref{eq}, if it satisfies \eqref{eq:static} in a {\em weak sense}, namely there holds the following equality
\begin{multline}\label{weaksol}
\ird \n u\cdot \n v\, dx +  \ird\frac{h_u^2(|x|)}{|x|^2}uv\,dx\\
+ \ird \left(\int_{|x|}^{+\infty}\frac {h_u(s)}{s}u^2(s)\,ds\right)uv\, dx =\ird |u|^{p-1}u v\,dx,
\end{multline}
for all $v$ in $\H$.
\end{definition}

Finally we give the definition of {\em classical solution}.
\begin{definition}\label{de:cla}
A {\em classical solution} of \eqref{eq} is a radial function $u\in C^2(\RD)$ such that
\[
U_u(x):=
\begin{cases}
\frac{h_{u}^{2}(|x|)}{|x|^2} & \hbox{if }x\neq 0,
\\
0& \hbox{if }x= 0,
\end{cases}
\]
and 
\[
V_u(x):=\int_{|x|}^{+\infty}\frac {h_u(s)}{s}u^2(s)\,ds
\]
are well defined and continuous in $\RD$, $u$ satisfies \eqref{eq:static} pointwise and goes to $0$ as $x$ goes to $\infty$.
\end{definition}

In Proposition \ref{pr:rem}, we will show that Definition \ref{def:ws} and Definition \ref{de:cla} coincide when the energy of the solution is finite, namely every $u\in \calh$ is weak solution of \eqref{eq} if and only if $u$ is a classical solution of \eqref{eq}.

In the Appendix \ref{appendix}, we will study sufficient integrability conditions on $u$ for $U_u$ and $V_u$ to be well defined on $\RD$.

We can state now our first result, which guarantees the existence of a static finite energy solution of system \eqref{eq:e0},  satisfying \eqref{ansalz} and \eqref{eq:a0}.
\begin{theorem}\label{th:main1}
For any $p>3$, there exists $u\in \calh$ classical positive solution of \eqref{eq}. \\
As a consequence the quadruplet $(\phi, A_0, A_1,A_2)$ defined as in \eqref{ansalz} for $\o=0$  is in $C^2(\RD)\times (C^1(\RD))^3$ and it is a static positive energy solution of \eqref{eq:e0} satisfying the following weak formulation of the large-distance fall-off requirement
	\begin{equation*}
			\lim_{|x|\to +\infty}A_0(x)=0,\quad A_1\in L^{\infty}(\RD),\quad A_2\in L^{\infty}(\RD).
	\end{equation*}
\end{theorem}

In the previous result, the positiveness of the energy is a consequence of Nehari and Pohozaev identities (see Proposition \ref{pr:posen}). We underline that the failure to use variational methods to find solutions
causes non-trivial difficuties in deducing these identities. In particular, the fundamental Nehari and Pohozaev identities are not immediately available by means of direct computations based on standard arguments as in \cite{BHS}, but they both require quite tricky ad-hoc strategies. 

These identities also play a key role in view of an analysis of the energy levels and in particular in order to estimate the {\em zero-point energy} of our system. The crucial question of establishing whether a ground state (at least limiting to static waves satisfying our ansatz) exists, translates into a minimum problem consisting in minimizing the functional of the energy in the set of solutions in $\calh$. Observe that, since by Theorem \ref{th:main1} the set
\begin{equation}\label{esse}
\S:=\{u\in \calh\setminus\{0\} :  u \hbox{ is a classical solution of \eqref{eq} } \}
\end{equation}
is not empty, and by positiveness of energy the set $\{E(u): u\in\S\}$ is bounded below, the minimizing problem makes sense.

Actually, we will prove that the infimum is attained.
\begin{theorem}\label{th:gs}
For any $p>3$, there exists a non-trivial radial {\em ground state}, namely there exists  $\bar u\in \S$ such that
\[
E(\bar u)=\inf_{u\in \S}E(u).
\]
\end{theorem}

As for the energy, the estimate of the total charge of our static wave presents analogous difficulties due to the particular zero mass structure of equation \eqref{eq:static}. 
In addition to evident problems related with the possibility that the total charge may be infinite, by \eqref{ansalz} this fact is reflected in \eqref{eq:a012} which is, in general, hard to verify. However, a priori considerations, based on a comparison argument, lead to the following (quite surprising) result
 
\begin{theorem}\label{th:ul2}
Assume that $p>9$ and let $u$ be the solution found in Theorem \ref{th:main1}. Then $u$ has finite total charge (that is $u$ is in $L^2(\RD)$) and the corresponding quadruplet $(\phi, A_0, A_1,A_2)$ is a positive energy static solution of \eqref{eq:e0} satisfying \eqref{eq:a012} .
\end{theorem}

This paper is organized as follows. 
\\
In Section \ref{se:ff}, we present the functional framework introducing some useful properties of the spaces $\ch$ and $\chr$.  
\\
Section \ref{se:3} is devoted to the most of the proof of Theorem \ref{th:main1} (positive energy of our static solution is a consequence of Proposition \ref{pr:posen} in Section \ref{se:gs}). Following \cite{APP,BBS}, as first step, roughly speaking we add a positive mass to the functional $E$; more precisely, for any $\e>0$, we consider the following {\em perturbed}  functional
\begin{equation*}
I_\e(u)=\frac 12 \|\n u\|_2^2+\frac \e 2\|u\|_2^2
+\frac 12 \ird  \frac{h^2_u u^2}{|x|^2} dx
-\frac{1}{p+1} \|u\|_{p+1}^{p+1},
\end{equation*}
defined in $\Hr$. By \cite{BHS}, it is easy to see that there exists a critical point $u_\e$ of $I_\e$, for any $\e>0$. The second step consists in studying the behaviour of the family $\{u_\e\}_{\e>0}$, as $\e\searrow 0$. By concentration-compactness arguments, we show that, up to a subsequence, there exists $u_0\in \calh$ such that the family converges weakly to such $u_0$ in $\chr$, as $\e\searrow 0$. This will be enough to prove that, actually, $u_0$ is the desired  solution.
\\
In Section \ref{se:gs}, we perform a deep analysis of the properties related with the energy of our static wave, and prove Theorem \ref{th:gs}. An interesting consequence of this study and the result in \cite{BHS} is the existence of a continuum of positive energy non-static standing waves stated in the Corollary \ref{chepalle}. Moreover, the existence of a ground state will be obtained, again by a concentration-compactness argument, by means of Nehari and Pohozaev identities holding for \eqref{eq}.
\\
Finally, in Section \ref{se:carica} we show that, when $p>9$, our static wave has finite total charge and Theorem \ref{th:ul2} holds. The proof is based on a contradiction argument and a precise estimate of the decay at infinity of the solution will play a crucial role.

We conclude this introduction fixing some notations. For any $\tau\ge 1$, we denote by $L^\tau(\R^2)$ the usual Lebesgue spaces equipped by the standard norm $\|\cdot\|_{\tau}$.
In our estimates, we will frequently denote by $C>0$, $c>0$ fixed
constants, that may change from line to line, but are always
independent of the variable under consideration. 
%We also use the notations $O(1), o(1), O(\e), o(\e)$ to describe the asymptotic behaviors of quantities in a standard way.
Moreover, for any $R>0$, we denote by $B_R$ the ball of $\RD$ centred in the origin with radius $R$.
Finally the letters
$x$, $y$ indicate two-dimensional variables and $r$, $s$ denote
one-dimensional variables.

\section{Functional framework}\label{se:ff}

In this section we introduce the functional framework presenting some useful properties of the spaces $\ch$ and $\chr$.

The following inequality will play an essential role in our arguments. It is essentially already contained in \cite{BHS}, where it is proved for $H^1_r(\RD)$ functions (see \cite[Proposition 2.4]{BHS}), but actually it holds also in $\chr$. The proof is based on the same density argument used in \cite{BHS} after having showed its validity in $C_0^\infty(\RD)$ and therefore we omit it.
\begin{proposition} \label{prop:dis} 
For any $u \in \chr$, the following inequality holds:
\begin{equation} \label{eq:5.1}
\|u\|^4_4
\le 4 \| \nabla u\|_2 \left( \ird \frac{h_u^2 u^2}{|x|^2} dx \right)^{\frac{1}{2}}.
\end{equation}
\end{proposition}

\begin{remark}
We observe that the right hand side in inequality \eqref{eq:5.1} could be also infinity, while it is surely finite if $u\in \chr$ with finite energy.
\end{remark}

\begin{proposition}\label{pr:ref}
$(\ch,\|\cdot\|_{2,4})$ is a reflexive Banach space.
\end{proposition}

\begin{proof}
To prove that the normed space is reflexive it is sufficient to observe that $\|\cdot\|_{2,4}$ is equivalent to $\|\cdot\|_*=\sqrt{\|\n\cdot\|_2^2+\|\cdot\|_4^2}$ and $(\ch,\|\cdot\|_*)$ is an uniformly convex normed space.\\
Now we prove it is complete.
Let $\{ u_n\}_n$ be a Cauchy sequence in $\ch$. Then $\{ u_n\}_n$ is a Cauchy sequence in $L^4(\RD)$ and $\{ \n u_n\}_n$ is a Cauchy sequence in $L^2(\RD)$. Since $L^4(\RD)$ is complete, there exists $u \in L^4(\RD)$ such that $\lim_n u_n = u$ in $L^4(\RD)$. Since $L^2(\RD)$ is complete, then there exists ${\bf U} \in L^2(\RD)$ such that $\lim_n \n u_n = {\bf U}$ in $L^2(\RD)$.
We want to prove that $\n u = {\bf U}$ in the distributions sense, i.e. that for every $\varphi \in C_0^\infty (\RD)$
\[
\int_\RD u \n \varphi \, dx = - \int_\RD \varphi {\bf U}\, dx.
\]
Obviously, for every $\varphi \in C_0^\infty (\RD)$ and for every $n\in\N$
\[
\int_\RD u_n \n \varphi \, dx= - \int_\RD \varphi \n u_n\, dx.
\]
So it is sufficient to prove that
\[
\lim_n \int_\RD u_n \n \varphi \, dx= \int_\RD u \n \varphi\, dx
\quad\hbox{ and }\quad
\lim_n \int_\RD \varphi \n u_n\, dx=\int_\RD \varphi {\bf U}\, dx.
\]
Indeed, since $\lim_n u_n = u$ in $L^4(\RD)$, then
\[
\left|\int_\RD  ( u_n - u)\n\varphi\, dx\right| \leq \|\n\varphi\|_{\frac 43} \|u_n - u\|_4 \to 0,
\]
while, since $\lim_n \n u_n = {\bf U}$ in $L^2(\RD)$ then 
\begin{equation*}
\left|\int_\RN \varphi (\n u_n - {\bf U})\, dx\right| 
 \le  \|\varphi\|_{2} \|\n u_n-{\bf U}\|_{2} \to 0.
\end{equation*}
\end{proof}

	\begin{proposition}\label{pr:density}
		The space $\ch$ corresponds to the set 
		$${\mathcal W}^{2,4}(\RD):=\{u\in L^4(\RD): \n u \in L^2(\RD)\}.$$
		Moreover, if we define 
			$${\mathcal W}_r^{2,4}(\RD)=\{u\in {\mathcal W}^{2,4}(\RD): u \hbox{ is radially symmetric}\},$$
		then $\chr={\mathcal W}_r^{2,4}(\RD)$.
	\end{proposition}
	\begin{proof}
			We have just to show that the functions in ${\mathcal W}^{2,4}(\RD)$ can be approximate in the norm $\|\cdot\|_{2,4}$ by functions in the same space, with compact support. The rest of the proof proceeds following standard arguments (see \cite[Theorem 7.6]{LL}). 
			\\
			Indeed, consider $u\in{\mathcal W}^{2,4}(\RD)$ and let $k:\RN\to[0,1]$ be a cut off smooth function such that $k\equiv 1$ in $|x|\le 1$ and $k\equiv 0$ in $|x|\ge 2.$
			For any $M>0$, define $v_M= k_M u $, where $k_M(x)=k(x/M),$ and set $A_M=\{x\in \RD: M\le |x|\le 2M\}$.
			Certainly $v_M$ has a compact support and it is in $L^4(\RD)$.\\ 
			Moreover, since $\n v_M= k_M\n u+u\n k_M$, of course $\n v_M \in L^2(\RD).$
			We easily have that
			\begin{equation*}%\label{eq:fg}
			\|u-v_M\|^4_4\le\int_{B_M^c}|u|^4\,dx =o_M(1),
			\end{equation*}
			where $o_M(1)$ denotes a vanishing function as $M\to+\infty.$
			\\
			Moreover
			\begin{align*}%\label{eq:nfg}
			\|\n u-\n v_M\|^2_2&\le C \int_{|x|\ge M}|\n u|^2\,dx+ \frac C{M^2}\int_{A_M} u^2\, dx\\
										&\le o_M(1)+ \frac C{M^2}\|u\|_4^2|A_M|^{\frac 12}\\
										&\le o_M(1)+ \frac C{M}\|u\|_4^2,
			\end{align*}
			and then we conclude.
	\end{proof}

In the following proposition we study the embedding's properties of $\ch$.
\begin{proposition}\label{pr:embedding}
The space $\ch$ is continuously embedded into $L^q(\RD)$, for any $q\in [4,+\infty)$.
\end{proposition}

\begin{proof}
Going back the proof of the Sobolev inequality, if $u\in C_0^{\infty}(\RD)$, one has
\begin{equation} \label{eq:a.1}
\| u\|_{2} \le 
\left\| \frac{\partial u}{\partial x_1} \right\|_1^{\frac{1}{2}}
\left\| \frac{\partial u}{\partial x_2} \right\|_1^{\frac{1}{2}}.
\end{equation}
See \cite[(19), P.\,280]{B}.
Let $m \ge 2$. Applying \eqref{eq:a.1} to $|u|^{m-1}u$, we get
$$
\| u\|_{2m}^m
\le C 
\left\| |u|^{m-1} \frac{\partial u}{\partial x_1} \right\|_1^{\frac{1}{2}}
\left\| |u|^{m-1} \frac{\partial u}{\partial x_2} \right\|_1^{\frac{1}{2}}
\le C \| \nabla u\|_2 \| u\|_{2(m-1)}^{m-1}.$$
By the Young inequality, it follows that
\begin{equation} \label{eq:a.2}
\| u\|_{2m} \le C(  \| \nabla u\|_2+\| u\|_{2(m-1)} ).
\end{equation}
In \eqref{eq:a.2}, we first choose $2(m-1)=4$, that is, 
$m= 3$. 
Thus from \eqref{eq:a.2}, we obtain
$$
\| u\|_{6}
\le C(  \| \nabla u\|_2+\| u\|_{4} )
=  C \| u\|_{2,4}.
$$
Iterating this procedure with $m=3+j$ for $j \in \mathbb{N}$,
and applying the interpolation inequality, one gets
$$
\| u\|_q \le C \| u\|_{2,4}
\quad \hbox{for all} \ u\in C_0^{\infty}(\RD) 
\ \hbox{and} \ q\in[4,+\infty).$$
This completes the proof by a density argument.

\end{proof}

\begin{remark}\label{re:loc}
It is easy to see that $\mathcal{H}^{2,4}_{\rm loc}(\RD)=H^{1,2}_{\rm loc}(\RD)$ and so $\mathcal{H}^{2,4}_{\rm loc}(\RD)$ is compactly embedded 
into $L^q_{\rm loc}(\RD)$, for any $q\in [1,+\infty)$.
\end{remark}

We now introduce a new Strauss Radial Lemma (see \cite{strauss}) in $\chr$.

	\begin{proposition}\label{pr:radial}
		For any $\tau\in\left(0,\frac 14\right)$, there exists $C_\tau>0$ and $R_\tau>0$ such that, for all  $u\in \chr$, we have 
			\begin{equation*}
|u(x)|\le C_\tau\frac{\|u\|_{2,4}}{|x|^\tau},\,\quad\hbox{  for } |x|\ge R_\tau.
			\end{equation*}  
	\end{proposition}
	\begin{proof}
		Let $k\in\left(0,\frac 12\right)$ and consider $u$ a radial function in $C_0^\infty(\RD)$.
		For any $r\ge 0$, we have that 
			\begin{align*}
				\left|\frac{d}{dr}\left(r^ku^2(r)\right)\right|&\le k r^{k-1}u^2(r)+2r^k|u(r)||u'(r)|\\
			 &\le k r^{k-1}u^2(r) + r^{2k-1}u^2(r) + r|u'(r)|^2. 
			\end{align*}
		Now, fix $r\ge 1$ and integrate $-\frac{d}{ds}\left(s^ku^2(s)\right)$ in the interval $[r,+\infty)$. We have
			\begin{align*}
				r^ku^2(r)&\le k\int_r^{+\infty} s^{k-\frac 32}s^{\frac 12}u^2(s)\, ds+\int_r^{+\infty}s^{2k-\frac 32}s^{\frac 12}u^2(s)\, ds+ \frac{\|\n u\|_2^2} {2\pi}\\
				&\le \frac k{\sqrt{2\pi}} \left(\int_r^{+\infty} s^{2k-3}\,ds\right)^{\frac12}\|u\|_4^2+\frac 1{\sqrt{2\pi}}\left(\int_r^{+\infty} s^{4k-3}\,ds\right)^{\frac 12}\|u\|_4^2+\frac{\|\n u\|_2^2} {2\pi}\\
				&\le C (r^{k-1}+r^{2k-1})\|u\|_4^2+\frac{\|\n u\|_2^2} {2\pi}
\le C\|u\|_{2,4}^2.
			\end{align*}
The conclusion follows easily by density arguments.
	\end{proof}

The following compact embedding result holds.

	\begin{proposition}\label{prcomp}
		The space $\chr$ is compactly embedded into $L^q(\RD)$, for any $q\in (4,+\infty)$.
	\end{proposition}
	\begin{proof}
		Taking into account Proposition \ref{pr:embedding} and Proposition \ref{pr:radial} the proof follows the same arguments as in \cite[Compactness Lemma 2]{strauss}.
	\end{proof}

\section{Existence of a static solution}\label{se:3}

First, we will study the following {\em perturbed} equation adding a positive  small mass term to \eqref{eq}. More precisely, for any $\e>0$ we consider
\begin{equation}\label{eq-eps}\tag{$\mathcal{P}_\e$}
\begin{cases}
\dis 
-\Delta u +\e u
+ \left(\frac{h_{u}^{2}(|x|)}{|x|^2} 
+ \int_{|x|}^{+\infty} \frac{h_u(s)}{s}  u^2(s) \,ds  \right)u =
|u|^{p-1}u \qquad \hbox{in } \R^2, 
\\[5mm]
u(x)\to 0, \qquad \hbox{as }|x| \to +\infty.
\end{cases}
\end{equation}
Solutions of \eqref{eq-eps} can be found as critical points of the functional 
\begin{equation*}
I_\e(u)=\frac 12 \|\n u\|_2^2+\frac \e 2\|u\|_2^2
+\frac 12 \ird  \frac{h^2_u u^2}{|x|^2} dx
-\frac{1}{p+1} \|u\|_{p+1}^{p+1},
\end{equation*}
which is well defined in classical Sobolev space 
$$\Hr:=\{u\in \H:u \hbox{ is radially symmetric} \}.$$

Following \cite{BHS}, we define a Pohozaev-Nehari type manifold
\begin{equation*}
\M_\e:=\{u\in \Hr\setminus\{0\} : J_\e (u)=0 \},
\end{equation*}
where 
\begin{equation*}%\label{eq:Jeps}
J_\e (u)=\a \|\n u\|_2^2+\e (\a-1)\|u\|_2^2
+(3\a -2) \ird  \frac{h^2_u u^2}{|x|^2} dx
-\frac{(p+1)\a -2}{p+1} \|u\|_{p+1}^{p+1},
\end{equation*}
and we have fixed $\a >1$ and such that $\frac{2}{p-1}<\a<\frac{2}{5-p}$, for $p\in (3,5)$ and $\a >1$ arbitrary, for $p\ge 5$.

We have the following
\begin{proposition}[\!\cite{BHS}]\label{pipi}
For any $\e>0$,  there exists
$u_\e\in \Hr$ which is a positive solution of \eqref{eq-eps} and such that
\[
I_\e(u_\e)=\inf_{u\in \M_\e}I_\e(u)=:m_\e>0.
\]
\end{proposition}

Moreover these minimum's levels are uniformly bounded by positive constants both from above and from below. Indeed we have
\begin{proposition}\label{me<C}
There exists $C>0$ such that for any $\e\in (0,1)$ we have $C\le m_\e\le m_1$.
\end{proposition}

\begin{proof}
In the following,  for every $w\in \Hr$,  we set  
	\begin{equation*}
		a(w):=\|\n w\|_2^2,\quad b(w):=\|w\|_2^2, \quad c(w):= \ird \frac{h_w^2 w^2}{|x|^2}dx.
	\end{equation*}
	Consider $u\in \M_1$ and for any $t>0$ assume the following notation $u_t:=t^\a u(t\cdot)$, where $\a$ is choosen as in the definition of $J_\eps$. If we denote by $t_\eps>0$ the unique value for which $J_\eps(u_{t_\eps})=0$  (see \cite{BHS}), by simple computations we see that $t_\eps<1$ for $\eps\in(0,1)$. 
Now,
we have that
	\begin{align*}
		m_\eps&\le I_\eps(u_{t_\eps})\\
			&=\left(\frac 12 -\frac{\a}{(p+1)\a-2}\right)a(u_{t_\eps})\\
			&\qquad+\eps\left(\frac 12-\frac{\a-1}{(p+1)\a-2}\right)b(u_{t_\eps})+\left(\frac 12-\frac{3\a-2}{(p+1)\a-2}\right)c(u_{t_\eps})\\
			&=\left(\frac 12 -\frac{\a}{(p+1)\a-2}\right)t_\eps^{2\a}a(u)\\
			&\qquad+\eps\left(\frac 12-\frac{\a-1}{(p+1)\a-2}\right)t_\eps^{2(\a-1)}b(u)+\left(\frac 12-\frac{3\a-2}{(p+1)\a-2}\right)t_\eps^{6\a-4}c(u)\\
			&\le\left(\frac 12 -\frac{\a}{(p+1)\a-2}\right)a(u)\\
			&\qquad+\eps\left(\frac 12-\frac{\a-1}{(p+1)\a-2}\right)b(u)+\left(\frac 12-\frac{3\a-2}{(p+1)\a-2}\right)c(u)\\
			&=I_1(u).
	\end{align*}
Passing to the infimum, we have $m_\e\le m_1$.
\\
Now suppose by contradiction that, for a suitable $\eps_n\to 0$, it results that $m_{\eps_n}\to 0$.  For any $n\in \N$, let  $u_n \in \M_{\e_n}$ such that $I_{\e_n}(u_n)=m_{\e_n}$. Then we have that
	\begin{equation}\label{eq:cont}
		a(u_n)\to 0\qquad\hbox{ and }\qquad c(u_n)\to 0.
	\end{equation}
Since $u_n\in \M_{\e_n}$, by Proposition \ref{pr:embedding} we have that, for suitable positive constants $C_1$ and $C_2$,
	\begin{equation}\label{eq:fir}
		a(u_n)+ c(u_n)\le C_1\|u_n\|_{p+1}^{p+1}\le C_2 \|u_n\|_{2,4}^{p+1}.			
	\end{equation}
On the other hand, by \eqref{eq:5.1} and taking into account that $a(u_n)\to 0$, for a suitable constant $C>0$, we obtain
	\begin{align}\label{eq:sec}
		\|u_n\|_{2,4}& = (a(u_n))^{\frac{1}2} + \|u_n\|_4 \le (a(u_n))^{\frac{1}2} +  \big(a(u_n) +8 c(u_n)\big)^{\frac 14}\\
		&\le 2 \big(a(u_n) +8 c(u_n)\big)^{\frac 14}\le C \big(a(u_n) + c(u_n)\big)^{\frac 14}.\nonumber
	\end{align}
Inequalities \eqref{eq:fir} and \eqref{eq:sec} contradict \eqref{eq:cont}.
\end{proof}

As an immediate consequence of Proposition \ref{me<C}, we have
\begin{proposition}\label{pr:bdd}
The family $\{u_\e\}_{\e>0}$ is bounded in $\ch$.
\end{proposition}

In the following we fix a decreasing sequence $\{\e_n\}_n$ which tends to zero as $n \to +\infty$.

We define 
\[
\begin{array}{ll}
\dis a_1:=\left(\frac 12 -\frac{(1+\d)\a}{(p+1)\a-2}\right), \; &\dis a_2^n:=\eps_n\left(\frac 12-\frac{(1+\d)(\a-1)}{(p+1)\a-2}\right),
\\[5mm]
\dis a_3:= \left(\frac 12-\frac{(1+\d)(3\a-2)}{(p+1)\a-2}\right),\; 
&\dis a_4:=\frac\d{p+1},
\end{array}
\]
observing that, for $\d>0$ small enough and 
\begin{equation*}
\begin{cases}
\dis \a \in \left( \frac{2}{p-1-2\d}, \frac{4\d+2}{5+6\d-p}\right),&\hbox{ if }
3<p\le 5\\[5mm]
\a >1,&\hbox{ if }
p > 5,
\end{cases}
\end{equation*}
$a_i>0$ for any $i=1,\ldots,4$.\\
For any $n\ge 1$ define $u_n:=u_{\e_n}$, where $u_{\e_n}$ is as in Proposition \ref{pipi}, 
	\begin{equation*}
		\nu_n(\O):= a_1\int_\O|\n u_n|^2\, dx+a_2^n\int_\O u_n^2\, dx
		+a_3\int_\O\frac{h_{u_n}^2u_n^2}{|x|^2} dx+a_4\int_\O u_n^{p+1}\, dx,
	\end{equation*}
for any measurable $\O\subset \R^2$,  and
	\begin{equation*}
		G_n(u):= a_1\ird|\n u|^2\, dx+a_2^n\ird u^2\, dx
		+a_3\ird\frac{h_u^2 u^2}{|x|^2} dx+a_4\ird |u|^{p+1}\,dx
	\end{equation*}
for any $u\in\Hr$.
Of course $\nu_n(\RD)=G_n(u_n)=I_{\e_n}(u_n)=m_{\e_n}=\inf_{u\in \M_{\e_n}}I_{\e_n}(u)$.\\
By Proposition \ref{me<C}, we assume that, up to a subsequence, 
	\begin{equation}\label{eq:mn}
\lim_n	\nu_n(\RD)=\lim_n m_{\e_n}= m >0.
	\end{equation}
By \cite{L1,L2} there are three possibilities:
	\begin{itemize}
		\item[1.] {\it concentration:} there exists a sequence $\{\xi_n\}_n$ in $\RD$ with the following property: for any $\epsilon> 0$, there exists $r = r(\epsilon) > 0$ such that
\[
\nu_n(B_r(\xi_n))\ge c-\epsilon;
\]
		\item[2.] {\it vanishing:} for all $r > 0$ we have that
\[
\lim_n \sup_{\xi\in\RD} \nu_n(B_r(\xi))=0;
\]
		\item[3.] {\it dichotomy:} there exist two sequences of positive measures $\{\nu_n^1\}_n$ and $\{\nu_n^2\}_n$, a positively diverging sequence of numbers $\{R_n\}_n,$ and $\tilde m \in (0,m)$ such that
			\begin{align*}
				&0\le \nu_n^1 + \nu_n^2\le \nu_n,\quad \nu_n^1(\RD)\to \tilde m,\quad \nu_n^2(\RD)\to m-\tilde m \\
				& {\rm Supp}\, \nu_n^1\subset B_{R_n},\quad {\rm Supp}\, \nu_n^2\subset B_{2R_n}^c.
			\end{align*}			
	\end{itemize}

\begin{proposition}\label{conc}
Concentration holds and, moreover, the sequence $\{\xi_n\}_n$ is bounded.
\end{proposition}

We preliminary prove the following two lemmas.
	\begin{lemma}
			Vanishing does not hold.
	\end{lemma}
	\begin{proof}
If vanishing held, then we would have that
\[
\lim_n \sup_{\xi\in\RD} \int_{B_r(\xi)} u_n^{p+1}=0.
\]
Since $p>3$, we have also that 
\[
\lim_n \sup_{\xi\in\RD} \int_{B_r(\xi)} u_n^{4}=0.
\]
Therefore, since by Proposition \ref{pr:bdd}, the sequence $\{u_n\}_n$ is bounded in $\ch$, by \cite[Lemma I.1]{L2}, we deduce that $u_n\to 0 $ in $L^{p+1}(\RD)$, as $n \to +\infty$, and so, being $J_{\e_n}(u_n)=0$, also $m_{\e_n}\to 0$, contradicting Proposition \ref{me<C}.
\end{proof}

	\begin{lemma}
			Dichotomy does not hold.
	\end{lemma}

	\begin{proof}
			As usual, we perform a proof by contradiction assuming that, on the contrary, dichotomy holds.\\
			Define $\rho_n\in C^1_0(\RD,[0,1])$ radial such that, for any $n\ge 1$,  $\rho_n\equiv 1$ in $B_{R_n}$, $\rho_n\equiv 0$ in $ B_{2R_n}^c$ and $\sup_{x\in\RD}|\n\rho_n(x)|\le \frac{2}{R_n}$. Moreover set $v_n=\rho_nu_n$ and $w_n=(1-\rho_n)u_n$, observing that $v_n, w_n\in \Hr$.\\
			Now we proceed by steps.
			
			{\it 1st step}: we prove that, defined $\O_n=\{x\in\RD: R_n\le |x|\le 2R_n\}$, we have
	\begin{equation}\label{eq:on}
		a_1\int_{\O_n}|\n z_n|^2\, dx+a_2^n\int_{\O_n} z_n^2\, dx+a_3\int_{\O_n}\frac{h_{z_n}^2 z_n^2}{|x|^2} dx+a_4\int_{\O_n}z_n^{p+1}\,dx\to 0,
	\end{equation}
for $z_n$ equal to $u_n$, $v_n$ and $w_n$.\\
Indeed observe that 
	\begin{align*}
		\nu_n(\O_n)&=m-\nu_n(B_{R_n})-\nu_n(B^c_{2R_n})+o_n(1)\\
							&\le m - \nu^1_n(B_{R_n})-\nu^2_n(B^c_{2R_n})+o_n(1)=o_n(1)
	\end{align*}
and then we deduce \eqref{eq:on} for $u_n$. \\
	By simple computations
		\begin{align*}
			&a_1\int_{\O_n}|\n v_n|^2\, dx+a_2^n\int_{\O_n} v_n^2\, dx
%\\			&\qquad\qquad
			+a_3\int_{\O_n}\frac{h_{v_n}^2v_n^2}{|x|^2} dx+a_4\int_{\O_n}v_n^{p+1}\,dx\\
			&\qquad\le 2 a_1\int_{\O_n}\left(|\n u_n|^2+\frac4{R^2_n} u_n^2\right)\, dx+a_2^n\int_{\O_n} u_n^2\, dx
%\\			&\qquad\qquad
+a_3\int_{\O_n}\frac{h_{u_n}^2 u_n^2}{|x|^2} dx+a_4\int_{\O_n}u_n^{p+1}\,dx
\\
			&\qquad\le \frac{8a_1}{R^2_n}\left(\int_{|x|\le 2R_n} 1\,dx\right)^{\frac 12}\|u_n\|_4^2 +o_n(1)\\
			&\qquad=\frac{16a_1\sqrt \pi}{R_n}\|u_n\|_4^2 +o_n(1)=o_n(1)
		\end{align*}
	and then we have proved \eqref{eq:on} also for $v_n$. The proof for $w_n$ is analogous.
	
	{\it 2nd step}: $\liminf_n G_n(v_n)=\tilde{m}$.\\
Observe, indeed, that since $h_{u_n}=h_{v_n}$ in $B_{R_n}$,  we have
\begin{equation}\label{superfico}
G_n(v_n)\ge \nu_n (B_{R_n})\ge \nu_n^1 (B_{R_n})\to \tilde{m},
\end{equation} 
%So we easily have that $\liminf_n G_n(v_n)\in[0,m]$. Assume by contradiction that $\liminf_n G_n(v_n)=0$. Then, again by the first step, we have that
%	 	\begin{equation*}
%			\liminf_n \left(a_1\int_{B_{R_n}}|\n v_n|^2\, dx+a_2^n\int_{B_{R_n}} v_n^2\, dx
%+a_3\int_{B_{R_n}}\frac{h_{v_n}^2 v_n^2}{|x|^2} dx +a_4\int_{B_{R_n}}v_n^{p+1}\,dx\right)=0.
%	 	\end{equation*}
%	Since $u_n=v_n$ in $B_{R_n}$, we deduce that
%		\begin{equation*}
%			\liminf_n\nu_n(B_{R_n})=0
%		\end{equation*}
%	and, since $\nu^1_n\le\nu_n$, 
%		\begin{equation*}
%			\liminf_n\nu_n^1(B_{R_n})=0.
%		\end{equation*}
%	Since ${\rm Supp}\, \nu_n^1\subset B_{R_n}$ and $\nu_n^1(\RD)\to \tilde m>0$, we get the contradiction.\\
	Now, observe that, by the first step and considering that $\nu_n\ge\nu_n^2$,
		\begin{align*}
			m&=\lim_n\nu_n(\RD)=\lim_n (\nu_n(B_{R_n})+\nu_n(B_{2R_n}^c))\\
			&\ge \liminf_n G_n(v_n)+\lim_n\nu_n^2(B_{2R_n}^c).
		\end{align*}
	Since $\lim_n\nu_n^2(\RD)=m-\tilde m$ and ${\rm Supp}\, \nu_n^2\subset B_{2R_n}^c$, we conclude that 
	$$\liminf_n G_n(v_n)=\tilde m.$$
	
	{\it 3rd step}: conclusion.\\
 First of all observe that, since $u_n=v_n+w_n$ and both $v_n$ and $w_n$ are nonnegative, then by the first step
		\begin{equation}\label{eq:Gun}
			G_n(u_n)\ge G_n(v_n)+G_n(w_n)+o_n(1). 
		\end{equation}
	Observe that, by step 1, 
		\begin{equation}\label{eq:Jun}
			0=J_{\e_n}(u_n)\ge J_{\e_n}(v_n)+J_{\e_n}(w_n)+o_n(1).
		\end{equation}
For any $n\in \N$, let $t_n, s_n>0$ be the numbers, respectively, such that $(v_n)_{t_n}\in \M_{\e_n}$ and  $(w_n)_{s_n}\in \M_{\e_n}$.

	There are three possibilities.

	{\it Case 1}:  up to a subsequence, $J_{\e_n}(v_n)\le 0$. 
\\
By simple computations we see that $t_n\le 1$ and then we have
				\begin{equation*}
					m_{\e_n}\le I_{\e_n}((v_n)_{t_n})=G_n((v_n)_{t_n})\le G_n(v_n)
				\end{equation*}
			which, for a large $n\ge 1$, leads to a contradiction due to the fact that, by \eqref{eq:mn} and step 2,
				$$\lim_n m_{\e_n}=m>\tilde m= \liminf_nG_n(v_n).$$

	{\it Case 2}:  up to a subsequence, $J_{\e_n}(w_n)\le 0.$\\
Then, proceeding as in the first case, by \eqref{superfico} and using \eqref{eq:Gun}, we have, for $n$ sufficiently large,
				\begin{equation*}
					m_{\e_n}\le I_{\e_n}((w_n)_{t_n})=G_n((w_n)_{t_n})\le G_n(w_n)\le G_n(u_n),
				\end{equation*}
			which, by \eqref{eq:mn}, implies $m=\lim_n G_n(w_n)$. Then, passing to the limit in \eqref{eq:Gun}, we have
				$$m\ge m + \liminf_nG_n(v_n)$$ 
			which contradicts the result obtained in step 2.

	{\it Case 3}:  there exists $n_0\ge 1$ such that for all $n\ge n_0$ both $J_{\e_n}(v_n)>0$ and $J_{\e_n}(w_n)>0$. 
\\
Then $\liminf_nt_n\ge1$ and, by \eqref{eq:Jun}, we also have that $J_{\e_n}(v_n)=o_n(1)$. \\
			If $ \liminf_n t_n = 1$, we can repeat the computations performed in the first case and get the contradiction. If $\liminf_n t_n >1$, from
				\begin{align*}
					o_n(1)& = J_{\e_n}(v_n)-\frac 1{t_n^{(p+1)\a-2}}J_{\e_n}((v_n)_{t_n})\\
					&=\a\left(1-\frac 1{t_n^{(p-1)\a-2}}\right)\|\n v_n\|_2^2+\eps_n(\a-1)\left(1-\frac 1{t_n^{(p-1)\a}}\right)\|v_n\|_2^2\\
					&\qquad+(3\a-2)\left(1-\frac 1{t_n^{(p-5)\a+2}}\right)\ird\frac{h_{v_n}^2 v_n^2}{|x|^2}dx
				\end{align*}
			we deduce that 
				\begin{align*}
					&\|\n v_n\|_2\to 0, \\
					&\eps_n\|v_n\|_2\to 0,\\
					&\ird\frac{h_{v_n}^2 v_n^2(x)}{|x|^2} dx\to 0
				\end{align*}
			and, as a consequence, also $\|v_n\|_{p+1}\to 0$ by Propositions \ref{prop:dis} and \ref{pr:embedding}. Of course, we get a contradiction since $\liminf_n G_n(v_n)>0$ by step 2.
	\end{proof}

\begin{proof}[Proof of Proposition \ref{conc}]
By the previous two lemmas we conclude that concentration holds. Moreover, the symmetry property of the functions $u_n$ guarantees the boundedness of $\{\xi_n\}_n$.
\end{proof}

The next two propositions provide fundamental integrability properties related to the nonlocal terms.

\begin{proposition}\label{pr:main}
   There exists $u_0\in  \chr$ such that, up to a subsequence, $u_n \weakto u_0$ in $ \ch$ and moreover
\begin{enumerate}[label=(\roman{*}), ref=\roman{*}]
\item \label{hlinfinito}$\frac{h_{u_0}}{|x|}\in L^\infty(\RD)$;
   		\item \label{hl2} $\frac{h_{u_0}}{|x|}u_0\in L^2(\RD),$ and
   			\begin{equation}\label{u011}
   				\frac{h_{u_n}}{|x|}u_n\to\frac{h_{u_0}}{|x|}u_0\qquad\hbox{ in }L^2(\RD);
   				%	\ird\frac{u_n^2(x)}{|x|^2}\left( \int_{0}^{|x|}su_n^2(s) \,ds\right)^2 dx\to\ird\frac{u_0^2(x)}{|x|^2}\left( \int_{0}^{|x|}su_0^2(s) \,ds\right)^2 dx,
   			\end{equation}
   		%\item[3.] $c'(u_n)[\vfi]\to c'(u_0)[\vfi]$, for any $\vfi\in C_0^\infty(\RD)$. 
   		\item \label{h2l2} $\frac{h^2_{u_0}}{|x|^2}u_0\in L^2(\RD)$;
\item \label{(hu2)u}  $V_{u_0}(x)=\dis\int_{|x|}^{+\infty}\frac {h_{u_0}(s)}{s}u_0^2(s)\,ds$ is well defined and continuous in $\RD$.
\end{enumerate}

\end{proposition}

\begin{proof}
	The existence of $u_0\in\ch$ is guaranteed by the fact that, since $\{G_n(u_n)\}_n$ is bounded, $\{u_n\}_n$ is bounded in $\chr$ and then it possesses a weakly convergent subsequence by Proposition \ref{pr:ref}.\\
	We can assume that such a sequence, relabeled $\{u_n\}_n$, is such that
		\begin{align*}
			&u_n\to u_0\hbox{ a.e. in }\RD\hbox{ (and then $u_0$ is radial and nonnegative)}\\
			&u_n\to u_0\hbox{ in }L^q(B), \hbox{ for all } B\subset\RD \hbox { bounded and } q\ge 1. 
		\end{align*}
To prove \eqref{hlinfinito}, observe that, for any $u\in L^4(\RD)$ and for any $x\in \RD\setminus \{0\}$, we have that
\[
\frac{h_u(x)}{|x|}=\frac{1}{4\pi|x|}\int_{B_{|x|}}u^2\, dy 
\le \frac{1}{4\pi|x|}\left(\int_{B_{|x|}} dy \right)^\frac 12
\left(\int_{B_{|x|}}u^4  dy \right)^\frac 12
\le C \|u\|_4^2.
\]
Therefore,
 since $u_0\in L^4(\RD)$ and $\{u_n\}_n$ is bounded in $L^4(\RD)$, we have
		\begin{equation}\label{linf}
			\frac{h_{u_0}}{|x|}\in L^{\infty}(\RD)\quad\hbox{ and }\quad \left\{\frac{h_{u_n}}{|x|}\right\}_n \hbox{ is bounded in } L^{\infty}(\RD).
		\end{equation}
	We prove \eqref{hl2}. 
	First of all we show that, for all $B\subset \RD $ bounded, we have
		\begin{equation}\label{eq:convbou}
			\int_B\left(\frac {h_{u_n}u_n-h_{u_0}u_0}{|x|}\right)^2\,dx\to 0.
		\end{equation}
	Indeed, since $u_n\to u_0$ in $L^2(B)$ for every $B\subset\RD$ bounded, we have that
		\begin{equation}\label{pointwise}
			h_{u_n}(x)\to h_{u_0}(x)\hbox{  for all }x\in\RD.
		\end{equation}
	By \eqref{linf}, \eqref{pointwise} and the dominated convergence theorem we obtain
		\begin{equation*}
			\int_B\left(\frac{h_{u_n}-h_{u_0}}{|x|}\right)^2u_0^2\,dx\to 0.
		\end{equation*}
Hence we deduce that
		\begin{align*}
			\int_B\left(\frac {h_{u_n}u_n-h_{u_0}u_0}{|x|}\right)^2\,dx&\le 2\left(	\int_B\frac{h^2_{u_n}}{|x|^2}(u_n-u_0)^2\,dx+\int_B\left(\frac{h_{u_n}-h_{u_0}}{|x|}\right)^2u_0^2\,dx\right)\\
			&\le \|h_{u_n}/|x|\|^2_{\infty}\|u_n-u_0\|_{L^2(B)}^2+o_n(1)
		\end{align*}
	and we obtain \eqref{eq:convbou}.
	\\
	By contradiction, suppose now that $\frac{h_{u_0}}{|x|}u_0\notin L^2(\RD)$. Then, for every $M\ge 0$, there exists $R>0$ such that
	\begin{equation*}
\int_{B_R}\frac{h_{u_0}^2u_0^2}{|x|^2}dx\ge M.
	\end{equation*}
In particular, there exists $R_m>0$ such that
	\begin{equation}\label{eq:ibr}
\int_{B_{R_m}}\frac{h_{u_0}^2u_0^2}{|x|^2}dx		\ge m+1
	\end{equation}
where  $m$ is defined in \eqref{eq:mn}.
By \eqref{eq:convbou} and \eqref{eq:ibr}, we get 
	\begin{equation*}
		\lim_n \int_{B_{R_m}}\frac{h_{u_n}^2u_n^2}{|x|^2}dx\ge m+1.
	\end{equation*} 
	which leads to a contradiction comparing with \eqref{eq:mn}.
	\\
Let us now prove that \eqref{u011} holds.\\
By Proposition \ref{conc}, we know that for any $\d>0$ there exists $R_\d>0$ such that uniformly for $n\ge 1$
	\begin{equation}\label{pippa}
		\int_{B_{R_\d}^c}\frac{h_{u_n}^2u_n^2}{|x|^2}\,dx
%{\color{red}\le C\int_{B_{R_\d}^c}\frac{u_n^2}{|x|^2}\left( \int_{0}^{|x|}su_n^2(s) \,ds\right)^2 dx}
\le \d.
	\end{equation}
Of course we can assume $R_\d$ large enough to have also
		\begin{equation}\label{sega}
				\int_{B_{R_\d}^c}\frac{h_{u_0}^2u_0^2}{|x|^2}\,dx
%{\color{red}\le C\int_{B_{R_\d}^c}\frac{u_0^2(x)}{|x|^2}\left( \int_{0}^{|x|}su_0^2(s) \,ds\right)^2 dx}
\le \d.
		\end{equation}
Then, by \eqref{eq:convbou}, we have
	\begin{align*}
\ird\left(\frac {h_{u_n}u_n-h_{u_0}u_0}{|x|}\right)^2\,dx&\le\int_{B_{R_\d}}\left(\frac {h_{u_n}u_n-h_{u_0}u_0}{|x|}\right)^2\,dx\\
&\qquad\qquad+2\left[\int_{B_{R_\d}^c}\frac{h_{u_n}^2u_n^2}{|x|^2}\,dx+\int_{B_{R_\d}^c}\frac{h_{u_0}^2u_0^2}{|x|^2}\,dx\right]\\
&\le o_n(1) +2\d
\end{align*}
and we conclude.
\\
The proof of \eqref{h2l2}, follows immediately by \eqref{hlinfinito} and \eqref{hl2}.
\\
Finally we prove \eqref{(hu2)u} showing that
\begin{equation}\label{leccaculo}
V_{u_0}(0)=\int_{0}^{+\infty}\frac {h_{u_0}(s)}{s}u_0^2(s)\,ds=\frac 1{2\pi}
\ird \frac{h_{u_0}}{|x|^2}u_0^2 \, dx \in \R,
\end{equation}
which implies also the continuity of $V_{u_0}$.
Observe that $\frac{u_0^2}{|x|}\in L^1(B_1)$. Indeed, we have
\[
\int_{B_1}\frac{u_0^2}{|x|}\, dx
\le\left(\int_{B_1}u_0^6\, dx\right)^\frac 13
\left(\int_{B_1}\frac{1}{|x|^\frac 32}\, dx\right)^\frac 23<+\infty.
\]
This, together with \eqref{hlinfinito}, implies that
\begin{equation}\label{lecca}
\int_{B_1}\frac{h_{u_0}}{|x|^2}u_0^2\, dx
\le \left\|\frac{h_{u_0}}{|x|}\right\|_\infty
\left\|\frac{u_0^2}{|x|}\right\|_{L^1(B_1)}<+\infty.
\end{equation}
Observe, moreover, that $\frac{u_0}{|x|}\in L^2(B_1^c)$. Indeed, we have
\[
\int_{B_1^c}\frac{u_0^2}{|x|^2}\, dx
\le\left(\int_{B_1^c}u_0^4\, dx\right)^\frac 12
\left(\int_{B_1^c}\frac{1}{|x|^4}\, dx\right)^\frac 12<+\infty.
\]
This, together with \eqref{hl2}, implies that
\begin{equation}\label{culo}
\int_{B_1^c}\frac{h_{u_0}}{|x|^2}u_0^2\, dx
\le \left\|\frac{h_{u_0}}{|x|}u_0\right\|_{L^2(B_1^c)}
\left\|\frac{u_0}{|x|}\right\|_{L^2(B_1^c)}<+\infty.
\end{equation}
Now \eqref{leccaculo} is a direct consequence of \eqref{lecca} and \eqref{culo}.
	\end{proof}

		\begin{proposition}\label{pr:finallemma}
			For every $v\in L^2(\RD)$ we have
			\begin{enumerate}[label=(\roman{*}), ref=\roman{*}]
				\item \label{h2uv}
				$\dis \ird  \frac{h_{u_n}^2}{|x|^2}u_nv\,dx\to \ird  \frac{h_{u_0}^2}{|x|^2}u_0v\,dx,$
				\item \label{h(uv)}  $ \dis \frac{h_{u_0}}{|x|^2}u_0^2\left( \int_{B_{|x|}}u_0 v \, dy \right)\in L^1(\RD)$ and 
				$$ \ird \frac{h_{u_n}}{|x|^2}u_n^2\left( \int_{B_{|x|}}u_n v \, dy \right) dx\to \ird \frac{h_{u_0}}{|x|^2}u_0^2\left( \int_{B_{|x|}}u_0 v \, dy \right)dx,$$
				\item \label{(hu2)uv}  $\dis \left(\int_{|x|}^{+\infty}\frac {h_{u_0}(s)}{s}u_0^2(s)\,ds\right)u_0 \in L^2(\RD)$ and
				$$2\pi\ird \left(\int_{|x|}^{+\infty}\frac {h_{u_0}(s)}{s}u_0^2(s)\,ds\right)u_0v\, dx=\ird\frac{h_{u_0}}{|x|^2}u_0^2\left( \int_{B_{|x|}}u_0 v  \, dy\right) dx.$$
			\end{enumerate}
		\end{proposition}
		\begin{proof}
			 Let $v\in L^2(\RD)$.
\\
			By \eqref{h2l2} of Proposition \ref{pr:main} we deduce that $ \frac{h_{u_0}^2}{|x|^2}u_0v\in L^1(\RD)$. Moreover, 
we prove easily \eqref{h2uv} if we show that
\begin{equation}\label{cicci}
\frac{h_{u_n}^2}{|x|^2}u_n \to \frac{h_{u_0}^2}{|x|^2}u_0 \quad \hbox{ in }L^2(\RD).
\end{equation}
Indeed, let $B$ a bounded domain in $\RD$, then by \eqref{linf},  \eqref{pointwise} and the dominated convergence theorem, we get
		\begin{equation*}
			\int_B\left(\frac{h_{u_n}^2-h_{u_0}^2}{|x|^2}\right)^2u_0^2\,dx\to 0.
		\end{equation*}
Hence we deduce that
		\begin{align*}
			\int_B\left(\frac {h_{u_n}^2u_n-h_{u_0}^2u_0}{|x|^2}\right)^2\,dx&\le 2\left(	\int_B\frac{h^4_{u_n}}{|x|^4}(u_n-u_0)^2\,dx+\int_B\left(\frac{h_{u_n}^2-h_{u_0}^2}{|x|^2}\right)^2u_0^2\,dx\right)\\
			&\le \|h_{u_n}/|x|\|^4_{\infty}\|u_n-u_0\|_{L^2(B)}^2+o_n(1).
		\end{align*}
Moreover, by \eqref{linf}, \eqref{pippa} and \eqref{sega}, we have
that, for any $\d>0$ there exists $R_\d>0$ such that, uniformly for $n\ge 1$,
	\begin{equation*}
		\int_{B_{R_\d}^c}\frac{h_{u_n}^4u_n^2}{|x|^4}\,dx
+\int_{B_{R_\d}^c}\frac{h_{u_0}^4u_0^2}{|x|^4}\,dx
\le \d.
		\end{equation*}
Therefore
	\begin{align*}
\ird\left(\frac {h_{u_n}^2u_n-h_{u_0}^2u_0}{|x|^2}\right)^2\,dx
&\le\int_{B_{R_\d}}\left(\frac {h_{u_n}^2u_n-h_{u_0}^2u_0}{|x|^2}\right)^2\,dx\\
&\qquad\qquad+2\left[\int_{B_{R_\d}^c}\frac{h_{u_n}^4u_n^2}{|x|^4}\,dx+\int_{B_{R_\d}^c}\frac{h_{u_0}^4u_0^2}{|x|^4}\,dx\right]\\
&\le o_n(1) +\d
\end{align*}
and we conclude the proof of \eqref{cicci}.
			\\
			Now we prove \eqref{h(uv)}. Observe that
				\begin{equation*}
					\ird\left|\frac{h_{u_0}}{|x|^2}u_0^2\left( \int_{B_{|x|}}u_0 v  \, dy\right)\right|\,dx
					\le C 	\ird\frac{(h_{u_0})^{\frac 32}}{|x|^2}u_0^2\,dx\,\|v\|_2.
%\\
%					&\le C \left[\int_{B_R}\frac{(h_{u_0})^{\frac 32}}{|x|^2}u_0^2\,dx+\int_{B_R^c}\frac{(h_{u_0})^{\frac 32}}{|x|^2}u_0^2\,dx\right]\|v\|_2.
				\end{equation*}
For $R>0$, we have
				\begin{align*}
					\int_{B_R}\frac{(h_{u_0})^{\frac 32}}{|x|^2}u_0^2\,dx&\le C \left(\|h_{u_0}/|x|\|_\infty^{\frac32}\int_{B_R}\frac{u_0^2}{|x|^{\frac12}}\,dx\right)\\
					&\le C \|h_{u_0}/|x|\|_\infty^{\frac32}\|u_0\|_4^2\left(\int_{B_R}\frac{1}{|x|}\,dx\right)^{\frac12}<+\infty
				\end{align*}
			while, taking into account the inequality $a^{\frac32}\le 1 +a^2$ that holds true for any $a\ge 0$,
				\begin{align*}
					\int_{B_R^c}\frac{(h_{u_0})^{\frac 32}}{|x|^2}u_0^2\,dx&\le\int_{B_R^c}\frac{u_0^2}{|x|^2}\,dx+\int_{B_R^c}\frac{h_{u_0}^2}{|x|^2}u_0^2\,dx\\
					&\le \|u_0\|_4^2\left(\int_{B_R^c}\frac{1}{|x|^4}\,dx\right)^{\frac12}+\int_{B_R^c}\frac{h_{u_0}^2}{|x|^2}u_0^2\,dx<+\infty
				\end{align*}
			due to \eqref{hl2} of Proposition \ref{pr:main}. We deduce, therefore, that $\frac{h_{u_0}}{|x|^2}u_0^2\left( \int_{B_{|x|}}u_0 v  \, dy\right)\in L^1(\RD)$.\\
Moreover, observe that, for any $R>0$,
\begin{align*}
& \ird\left| \frac{h_{u_n}}{|x|^2}u_n^2\left( \int_{B_{|x|}}u_n v \, dy \right) dx- \frac{h_{u_0}}{|x|^2}u_0^2\left( \int_{B_{|x|}}u_0 v \, dy \right)\right|dx\\
&\qquad\le  \int_{B_{R}}  |u_n^2-u_0^2|\frac{h_{u_n}}{|x|^2}\left( \int_{B_{|x|}}u_n |v|  dy\right)\!dx
+ \int_{B_{R}}  u_0^2\left| \frac{h_{u_n}-h_{u_0}}{|x|^2} \right|\left(\int_{B_{|x|}}u_n |v|   dy\right)\!dx\\
&\qquad\qquad+ \int_{B_{R}}  u_0^2 \frac{h_{u_0}}{|x|^2}\left( \int_{B_{|x|}}|u_n -u_0||v|  dy \right)\!dx
\\
&\qquad\qquad+\int_{B_{R}^c}  \frac{h_{u_n}}{|x|^2}u_n^2\left( \int_{B_{|x|}}u_n |v| dy \right)\!dx
+\int_{B_{R}^c}  \frac{h_{u_0}}{|x|^2}u_0^2\left( \int_{B_{|x|}}u_0 |v| dy\right) \!dx
\\
&\qquad=B_n^1+B_n^2+B_n^3+B_n^4+B^5.
\end{align*}
Now, $B_n^1\to 0$ by compact embedding in bounded domain and a proper application of H\"older inequality, whereas $B_n^2$ and $B_n^3$ go to zero by dominated convergence, again using properly the H\"older inequality (the scheme of the proof is similar to that used to obtain \eqref{u011}).\\
As to $B_n^4$, observe that by Proposition \ref{conc}, for $\d>0$ we can take $R>0$ such that 
\begin{equation}\label{eq:conc2}
\int_{B_R^c}\frac{h_{u_n}^2}{|x|^2}u_n^2\,d x< \d\quad\hbox{ and }\quad\sup_n\|u_n\|_4^4\int_{B_R^c}\frac{1}{|x|^4}\,dx<\d^2
\end{equation}
uniformly for $n\ge 1$.
Since for every $a\ge 0$ we know that $a^{\frac 32}\le  1 +a^2$, by Holder and \eqref{eq:conc2},
\begin{align*}
B_n^4&=\int_{B_{R}^c}  \frac{h_{u_n}}{|x|^2}u_n^2\left( \int_{B_{|x|}}u_n |v|  dy\right)\!dx\\
&\le  C\left[\int_{B_{R}^c}  \frac{(h_{u_n})^{\frac 32}}{|x|^2}u_n^2\,dx\right]\|v\|_2\\
&\le C\left[ \|u_n\|_4^2\left(\int_{B_R^c}\frac{1}{|x|^4}\,dx\right)^{\frac12}+\int_{B_R^c}\frac{h_{u_n}^2}{|x|^2}u_n^2\,dx\right]\|v\|_2<2\d \|v\|_2.
\end{align*}
Finally we prove that, for $R$ large enough, $B^5$ is less then $\d$ arguing as for $B_n^4$ and taking into account that  $\frac{h_{u_0}}{|x|^2}u_0^2\left( \int_{B_{|x|}}u_0 |v|dy\right)\in L^1(\RD)$.
\\
			As to \eqref{(hu2)uv}, observe that %since
			%	\begin{align*}
			%	\ird\frac{h_{u_0}}{|x|^2}u_0^2\left( \int_{B_{|x|}}u_0 v  \right) dx&= \int_0^{+\infty}\left(\frac{h_{u_0}(r)}{r}u_0^2(r)\int_0^{r}su_0(s) v(s)\, ds  \right)dr\\
				%&= \int_0^{+\infty}\!\!\!\left(\int_0^{+\infty}\frac{s}{r}\chi_{s<r}h_{u_0}(r)u_0^2(r)u_0(s) v(s)\, ds  \right)dr
				%\end{align*}
			we only have to prove that we can apply Fubini-Tonelli Theorem to the function $f:\RD\times \RD \to \R$, where for almost every $(x,y)\in \RD \times\RD,$ 
				\begin{equation*}
					f(x,y):=\frac{1}{|x|^2}\chi_{|y|<|x|}h_{u_0}(x)u_0^2(x)u_0(y) v(y).
				\end{equation*}
			It is easy to see that $f$ is measurable in $\R^4$ endowed with the product measure of $\RD$-Lebesgue measures.\\ 
			Moreover, denoted by $g(x):=\ird f(x,y)\,dy$ and by $\tilde g(x):=\ird |f(x,y)|\,dy$ we have
				\begin{equation*}
						\ird \tilde g(x)\,dx=\ird\left(\frac{h_{u_0}(x)}{|x|^2}u_0^2(x)\int_{B_{|x|}}u_0(y) |v(y)|\, dy  \right)dx<+\infty
				\end{equation*}
			by \eqref{h(uv)}. Then, by Fubini-Tonelli Theorem, for almost every $y\in\RD$ there exists $k(y):=\ird f(x,y)\, dx$. Moreover $k(y)\in L^1(\RD)$  and 
			$$\ird k(y)\,dy=\ird g(x)\,dx.$$
			It is easy to check that this corresponds exactly to what we claimed in \eqref{(hu2)uv}.
		\end{proof}

Now we can prove Theorem \ref{th:main1}, except the positivity of the energy of the solution, which will be a direct consequence of Proposition \ref{pr:posen}.  

			\begin{proof}[Proof of Theorem \ref{th:main1}]
			By Proposition \ref{pipi}, for any $n\in \N$, there exists  $u_n \in \Hr$ such that $u_n>0$ and $I_{\e_n}'(u_n)=0$ in $H^{-1}$. Hence, for every $v\in\H$, we have that $I_{\e_n}'(u_n)[v]=0$, namely
					\begin{multline*}
						\ird \n u_n \cdot\n v\,dx+ \eps_n \int u_n v\,dx + \ird  \frac{h_{u_n}^2}{|x|^2}u_nv\,dx\\
						+ \frac 1{2\pi}\ird \frac{h_{u_n}}{|x|^2}u_n^2\left( \int_{B_{|x|}}u_n v  dy\right) dx=\ird u_n^p v\, dx. 
					\end{multline*}
				By Proposition \ref{pr:main} there exists $u_0\in \chr$ such that, up to a subsequence, $u_n\rightharpoonup u_0$ in $\ch$. \\
It is immediate that $u_0\ge 0$. Moreover $\ird \n u_n\cdot\n v\,dx\to\ird \n u_0\cdot\n v\,dx$ and, by boundedness of $\sqrt{\eps_n} u_n$ in $L^2(\RD)$, we also deduce that
					\begin{equation*}
						\eps_n \int u_n v\,dx\le \sqrt{\eps_n} \|\sqrt{\eps_n} u_n\|_2\|v\|_2\to 0.
					\end{equation*}
				By compact embedding of $\chr$ into $L^q(\RD)$ for $q>4$  (see Proposition \ref{prcomp}), we also have $u_n^p\to u_0^p$ in $L^{\frac{p+1}{p}}(\RD)$ and then
					\begin{equation*}
						\left| \ird u_n^p v\,dx-\ird u_0^p v\,dx\right|\le \|u_n^p-u_0^p\|_{\frac{p+1}{p}}\|v\|_{p+1}\to 0.
					\end{equation*}
By Proposition \ref{pr:finallemma}, we conclude that \eqref{weaksol} holds, namely
 $u_0$ is a weak solution of \eqref{eq}. By \eqref{hlinfinito} and \eqref{(hu2)u} of Proposition \ref{pr:main} and by \cite[Theorem 8.8]{GT} we infer that $u_0\in W^{2,2}_{\rm loc}(\RD)$ and so $u_0\in C(\RD)$. Observing that the conclusions of \cite[Proposition 2.1]{BHS} hold for $u_0$, by bootstraps arguments, following again \cite{GT}, we conclude that $u\in C^2(\RD)$ and $u>0$ by the maximum principle. 
\\
Keeping in mind that  $A_0\in L^\infty(\RD)$ by Proposition \ref{pr:main}, we can show that $A_i\in C^1(\RD)$, for $i=0,1,2$, arguing as in \cite[Proposition 2.1]{BHS}. Finally the potentials verify the weak formulation of the large-distance fall-off requirement by \eqref{hlinfinito} and \eqref{(hu2)u} in Proposition \ref{pr:main}.
			\end{proof}

We conclude this section showing that the definitions of weak solutions and classical solutions  coincide for finite energy functions. More precisely the following holds.
\begin{proposition}\label{pr:rem}
Let $u\in \calh$. Then $u$ is weak solution of \eqref{eq} if and only if $u$ is a classical solution of \eqref{eq}.
\end{proposition}

\begin{proof}
Observing that all the integrability conditions of Propositions \ref{pr:main} and \ref{pr:finallemma} hold for functions belonging to $\calh$, then, arguing as in the last part of proof of Theorem \ref{th:main1}, we conclude.
\end{proof}

\section{Energy of static solutions}\label{se:gs}

We now prove that any weak solution with finite energy in the sense of Definition \ref{def:ws} satisfies a Nehari type identity. We would like to remark that this fact cannot be deduced as a trivial consequence of \eqref{weaksol} since, in general, we do not know if a weak solution is in $\H$. Moreover, while, in general,  the Nehari  identity is given by $E'(u)[u]=0$, in  our case, not only the weak solution is not found as a critical point of the functional but also the functional could be not well defined on the weak solution.
\begin{proposition}\label{pr:nehari}
Let $u\in \calh$ be a weak solution of \eqref{eq}, then it satisfies the following Nehari type identity
\begin{equation}\label{nehari}
\|\n u\|_2^2
+3 \ird  \frac{h_u^2 u^2}{|x|^2}\ dx
=\|u\|_{p+1}^{p+1}.
\end{equation} 
\end{proposition}

\begin{proof}
For any  $n\in \N$, let $\psi_n :\RD \to \R$, where
\[
\psi_n(x):=
\begin{cases}
1 & \hbox{if }|x|\le n,
\\[3mm]
\dis \frac{2n -|x|}{n} & \hbox{if  }n\le |x|\le 2n,
\\[3mm]
0 & \hbox{if }|x|\ge 2n.
\end{cases}
\] 
Being $\psi_n u\in \H$, for any $n\in \N$, we have that
\begin{equation}\label{nehar-n}
\begin{split}
&\ird \n u \cdot\n (\psi_n u)\, dx 
+  \ird \psi_n \frac{h_u^2 u^2}{|x|^2} \,dx
+ \ird \left(\int_{|x|}^{+\infty}\frac {h_u(s)}{s}u^2(s)\,ds\right)\psi_n u^2\, dx 
\\
&\qquad\qquad\qquad=\ird \psi_n |u|^{p+1} \,dx.
\end{split}\end{equation}
Observe that, being $u\in \chr$,
\begin{equation*}%\label{upsi}
\begin{split}
&\left|\int_{\R^2}\nabla u \cdot \n (\psi_n u)\, dx-
\int_{\R^2}|\nabla u|^2\, dx\right|\\
&\qquad\le
\int_{\R^2}|\nabla u |^2 |\psi_n -1|\, dx
+\int_{\R^2}|\nabla u||u|  |\n \psi_n|\, dx \\
&\qquad\le
\int_{B_n^c}|\nabla u |^2 \, dx
+\Big(\int_{B_n^c}|\nabla u|^2\, dx\Big)^\frac 12
\Big(\int_{B_n^c}|u|^{4}\, dx\Big)^\frac 1{4}
\Big(\int_{A_n}|\nabla \psi_n|^4\, dx\Big)^\frac 14
\\
&\qquad=o_n(1),
\end{split}
\end{equation*}
where $A_n:= B_{2n}\setminus B_n$.
\\
Analogously, being $u$ with finite energy and $u\in L^{p+1}(\RD)$, we have easily that
\begin{align}
\left| \ird \psi_n \frac{h_u^2 u^2}{|x|^2}\,dx
-  \ird\frac{h_u^2 u^2}{|x|^2}\,dx\right|
&=o_n(1), \label{hpsi}
\\
\left|\ird \psi_n |u|^{p+1} \,dx
-\ird  |u|^{p+1} \,dx\right| 
&=o_n(1). \label{p+1psi}
\end{align}
Finally observe that, due to the fact that $u$ has finite energy, arguing as in Proposition \ref{pr:finallemma}, we have that 
\begin{equation*}%\label{intint}
\ird \left(\int_{|x|}^{+\infty}\frac {h_{u}(s)}{s}u^2(s)\,ds\right)\psi_n u^2\, dx
=\frac1{2\pi}\ird\frac{h_{u}u^2}{|x|^2}\left( \int_{B_{|x|}}\psi_n u^2 dy \right) dx.
\end{equation*}
Therefore, using again the fact that $u$ has finite energy, we have
\begin{equation}\label{h2psi}
\begin{split}
&\left|
\ird \left(\int_{|x|}^{+\infty}\frac {h_{u}(s)}{s}u^2(s)\,ds\right)\psi_n u^2\, dx
-2 \ird\frac{h_u^2 u^2}{|x|^2}\,dx
\right|
\\
&\qquad=\left|
\frac1{2\pi}\ird\frac{h_{u}u^2}{|x|^2}\left( \int_{B_{|x|}}\psi_n u^2  dy\right) dx
-2 \ird\frac{h_u^2 u^2}{|x|^2}\,dx
\right|
\\
&\qquad=\left|
\frac1{2\pi}\ird\frac{h_{u}u^2}{|x|^2}\left( \int_{B_{|x|}}\psi_n u^2 dy \right) dx
-\frac1{2\pi}\ird\frac{h_{u}u^2}{|x|^2}\left( \int_{B_{|x|}}u^2 dy \right) dx
\right|
\\
&\qquad=
\frac1{2\pi}\ird\frac{h_{u}u^2}{|x|^2}\left( \int_{B_{|x|}}(1-\psi_n) u^2  dy\right) dx
\\
&\qquad=
\frac1{2\pi}\int_{B_n^c}\frac{h_{u}u^2}{|x|^2}\left( \int_{B_{|x|}}(1-\psi_n) u^2 dy \right) dx
\\
&\qquad\le 
2\int_{B_n^c}\frac{h_{u}^2u^2}{|x|^2} dx
=o_n(1).
\end{split}
\end{equation}
Now the conclusion follows by \eqref{nehar-n} together with \eqref{hpsi}, \eqref{p+1psi}, and \eqref{h2psi}.
\end{proof}

We now prove that each classical solution of \eqref{eq} with finite energy satisfies a Pohozaev type identity. We point out that even if a similar identity is present also in \cite{BHS}, we have to provide a different proof since their arguments need the essential information that the solution belongs to $L^2(\RD)$. Hence a new and different strategy is necessary.

	\begin{proposition}\label{pr:poho}
		Let $u\in \calh$ be a classical solution of \eqref{eq}, then $u$ satisfies the following Pohozaev type identity
			\begin{equation}\label{poho}
				 \ird  \frac{h_u^2 u^2}{|x|^2}\ dx
				=\frac 1 {p+1}\|u\|_{p+1}^{p+1}.
			\end{equation}
	\end{proposition}
	\begin{proof}
Let $u\in \calh$ be a classical solution of \eqref{eq} and fix $R>0$.  Multiplying by $\n u \cdot x$ and integrating by parts on $B_R$ we have
\begin{multline}
-\inb \Delta u (\n u \cdot x)\, dx
+\inb\frac{h_u^2}{|x|^2}u (\n  u\cdot x)\,dx
+ \inb \left(\int_{|x|}^{+\infty}\frac {h_u}{s}u^2(s)\,ds\right)u (\n  u\cdot x)\, dx \label{poho1}
\\
=\inb |u|^{p-1}u (\n u \cdot x)\, dx.
\end{multline}
Arguing as in \cite{BHS}, we infer that
\begin{align}
\inb \Delta u (\n u \cdot x)\, dx&=o_R(1), \label{poho2}
\\
\inb |u|^{p-1}u (\n u \cdot x)\, dx
&=-\frac{2}{p+1} \|u\|^{p+1}_{p+1}+o_R(1), \label{poho3}
\end{align}
where $o_R(1)$ denotes a vanishing function as $R\to+\infty.$
\\
Observe that we cannot repeat the arguments of \cite{BHS} to study also the remaining terms, because in their arguments it is essential the fact that $u$ belongs to $L^2(\RD)$. Therefore, we use another approach which seems, actually, less involved than that of \cite{BHS}. Integrating by parts, we have
\begin{equation} \label{poho4}
\begin{split}
&\inb\frac{h_u^2}{|x|^2}u (\n  u\cdot x)\,dx
+ \inb \left(\int_{|x|}^{+\infty}\frac {h_u(s)}{s}u^2(s)\,ds\right)u (\n  u\cdot x)\, dx \\
& \qquad = 2\pi\int_0^R h^2_uuu' \, dr + 2\pi\int_0^R \left(\int_r^{+\infty}\frac {h_u(s)}{s}u^2(s)\,ds\right)u  u' r^2\, dr\\
&\qquad= \pi h^2_u(R) u^2(R) - \pi \int_0^R h_u u^4 r\,dr\\ &\qquad\qquad+ \pi \left(\int_R^{+\infty}\frac {h_u(s)}{s}u^2(s)\,ds\right) u^2(R)R^2 + \pi \int_0^R h_uu^4 r\,dr\\
&\qquad\qquad\qquad- 2\pi\int_0^R \left(\int_r^{+\infty}\frac {h_u(s)}{s}u^2(s)\,ds\right)u^2 r\, dr.
\end{split}
\end{equation}
Being $u$ with finite energy, as observed in \cite{BL}, we have
\[
\liminf_{R\to +\infty}R\int_{\de B_R} \frac{h_u^2(|x|)}{|x|^2}u^2\, dx=0,
\]
and so, by radial symmetry, 
\[
\liminf_{R\to +\infty} h^2_u(R) u^2(R)=0.
\]
Using again the fact that $u$ has finite energy, by Fubini-Tonelli Theorem we deduce that $\left(\int_{|x|}^{+\infty}\frac {h_u(s)}{s}u^2(s)\,ds\right)u^2$ is in $L^1(\RD)$, since 
\[
\ird \left(\int_{|x|}^{+\infty}\frac {h_u(s)}{s}u^2(s)\,ds\right)u^2\, dx=2\ird  \frac{h_u^2(|x|)}{|x|^2}u^2\, dx.
\]
Hence, arguing as before, we have 
\[
\liminf_{R\to+\infty}\left(\int_R^{+\infty}\frac {h_u(s)}{s}u^2(s)\,ds\right) u^2(R)R^2 =0.
\]
Finally, another immediate consequence of the fact that $\left(\int_{|x|}^{+\infty}\frac {h_u(s)}{s}u^2(s)\,ds\right)u^2$ is in $L^1(\RD)$, we have that 
\begin{equation*}
\begin{split}
2\pi\int_0^R \left(\int_r^{+\infty}\frac {h_u(s)}{s}u^2(s)\,ds\right)u^2 r\, dr
&=\inb \left(\int_{|x|}^{+\infty}\frac {h_u(s)}{s}u^2(s)\,ds\right)u^2 \, dx
\\
&=\ird \left(\int_{|x|}^{+\infty}\frac {h_u(s)}{s}u^2(s)\,ds\right)u^2\, dx
+o_R(1)
\\
&=2\ird  \frac{h_u^2(|x|)}{|x|^2}u^2\, dx
+o_R(1).
\end{split}
\end{equation*}
By this, considering a suitable diverging sequence $\{R_n\}_n$, we conclude taking into account \eqref{poho1}, \eqref{poho2}, \eqref{poho3}, and \eqref{poho4}.
\end{proof}

%We consider the following subset of  solutions of \eqref{eq}, 
%\begin{equation*}
%\S=\{u\in \calh\cap C^2(\RD)\setminus\{0\} :  u \hbox{ is a classical solution of \eqref{eq} } \}.
%\end{equation*}
Recalling the definition of $\S$ given in \eqref{esse}, observe that, by \eqref{nehari} and \eqref{poho}, any $u\in \S$ satisfies 
\begin{equation}\label{neharipoho}
\a \|\n u\|_2^2
+(3\a -2) \ird  \frac{h^2_u u^2}{|x|^2} dx
-\frac{(p+1)\a -2}{p+1} \|u\|_{p+1}^{p+1}=0,
\end{equation}
where we have fixed $\a >1$ and such that $\frac{2}{p-1}<\a<\frac{2}{5-p}$, for $p\in (3,5)$ and $\a >1$ arbitrary, for $p\ge 5$. Moreover we have that the functional $E$ is well defined in $\S$.

 \begin{proposition}\label{pr:posen}
Every static finite energy solution of the form \eqref{ansalz} generated by $u\in\S$ has positive energy. Moreover
 we have that $\inf_{u\in \S}E(u)>0$.
 \end{proposition}

\begin{proof}
By Theorem \ref{th:main1}, we know that $\S$ is not empty.
\\
Now, if we compute $E$ on $\S$, we have 
	\begin{equation}\label{eq:I}
		E(u)=\left(\frac 12 -\frac \a{(p+1)\a -2}\right)\|\n u\|_2^2+\left(\frac 1 2 -\frac{3\a -2} {(p+1)\a -2}\right)\ird  \frac{h_u^2u^2}{|x|^2} dx
	\end{equation}
and then, by the choice of $\a $, for any  $p>3$, we have that  $\inf_{u\in\S}E(u)\ge 0$.\\
Assume by contradiction that, for a suitable sequence $\{u_n\}_n$ in $\S$, we have $E(u_n)\to 0$, then, by \eqref{eq:5.1}, we deduce also that $u_n\to 0$ in $\ch$.
\\ 
Using again \eqref{eq:5.1}, we have, moreover, that
\begin{equation*}
\|u_n\|^4_4
\le   C\left(\| \nabla u_n\|_2^2 
+ \ird \frac{h_{u_n}^2u_n^2}{|x|^2} dx\right)
\end{equation*}
and then, since $u_n$ satisfies \eqref{neharipoho}, we have
\begin{align*}
\|\n u_n\|_2^2
+\| u_n\|_4^4
\le C \|u_n\|_{p+1}^{p+1}.
\end{align*}
Therefore, taking into account that $\|u_n\|_{2,4}\to 0$ and by the continuous embedding $\ch\hookrightarrow L^{p+1}(\RD)$, we have that, for any $n\in\N$ large enough,
\[
\|u_n\|_{2,4}^4
\le C( \|\n u_n\|_2^2
+\| u_n\|_4^4)
\le C\|u_n\|_{p+1}^{p+1}\le C \|u_n\|_{2,4}^{p+1},
\]
which contradicts the fact that $u_n\to 0$ in $\ch$.
\end{proof}

As by-product of our results, we now prove the existence of positive energy non-static solution of \eqref{eq:e0} satisfying the ansatz \eqref{ansalz} with sufficiently small frequency.

\begin{corollary}\label{chepalle}
There exists $\o_0>0$ such that, for all $\o\in (0,\o_0)$, there exists  $(\phi,A_0,A_1,A_2)$, a positive energy non-static solution of \eqref{eq:e0} satisfying the ansatz \eqref{ansalz}.
\end{corollary}

\begin{proof}
Suppose by contradiction that that there exists a decreasing sequence $\{\o_n\}_n$ which tends to zero as $n \to +\infty$ and, for any $n\ge 1$, we define $u_n:=u_{\o_n}$, where $u_{\o_n}$ is as in Proposition \ref{pipi} and with $E(u_n)\le 0$. By Proposition \ref{pr:bdd} we infer that $\{u_n\}_n$ is bounded in $\ch$ and there exists $u_0\in \calh$ the weak limit of $\{u_n\}_n$ in $\ch$. Arguing as in the previous section we deduce that $u_0$ is a solution of \eqref{eq} which has positive energy by Proposition \ref{pr:posen} and such that the conclusions of  Proposition \ref{pr:main} hold. Then, by the weak lower semicontinuity of the norm, by the compact embedding of $\chr$ into $L^{p+1}(\RD)$ and by \eqref{u011}, we have
\[
0<E(u_0)\le \liminf_n E(u_n)\le 0,
\] 
reaching a contradiction.
\end{proof}

Now we have all the tools to conclude the prove Theorem \ref{th:gs}.

\begin{proof}[Proof of Theorem \ref{th:gs}]
	Consider $\{u_n\}_n$ a sequence in $\S$ such that $E(u_n)\to \inf_{u\in \S}E(u)$. By \eqref{eq:I}, the sequence is bounded in $\chr$ and then there exists $\bar u\in \ch$ such that, up to a subsequence, $u_n\rightharpoonup \bar u$ in $\ch$ and 
		\begin{align}
			u_n &\to \bar u\quad\hbox{ in }L^{p+1}(\RD),\label{eq:convstr}\\
			u_n &\to \bar u\quad\hbox{ in }L^{q}(B),\hbox{ for all } B\subset\RD \hbox { bounded and } q\ge 1,\label{eq:convboun}\\
			u_n &\to \bar u \quad\hbox{ a.e. in }\RD.\label{eq:qo}
		\end{align} 
	Of course $\bar u\in \chr$.\\
Arguing as in Section \ref{se:3}, we can see that also the minimizing sequence $\{u_n\}_n$ concentrates in the sense of \cite{L1,L2} and, arguing as in Propositions \ref{pr:main} and \ref{pr:finallemma}, this implies that $\bar u$ is a classical solution of \eqref{eq} with finite energy and so it satisfies \eqref{neharipoho}.
\\
By \eqref{neharipoho} and \eqref{eq:convstr}, therefore, we have that
\begin{equation}\label{eq:str}
\begin{split}
&\lim_n \left(\a \|\n u_n\|_2^2
+(3\a -2)\ird  \frac{h_{u_n}^2u_n^2}{|x|^2}dx\right)
=\frac{(p+1)\a -2}{p+1}\lim_n  \|u_n\|_{p+1}^{p+1}
\\
&\qquad\qquad\qquad\qquad=\frac{(p+1)\a -2}{p+1}\|\bar u\|_{p+1}^{p+1}=
\a\|\n \bar u\|_2^2
+(3\a -2)\ird  \frac{h_{\bar u}^2\bar u^2}{|x|^2} dx.
\end{split}
\end{equation}
Since  $\ird  \frac{h_{u_n}^2u_n^2}{|x|^2} dx$ is bounded, we can assume that, up to a subsequence, it is convergent.\\
We prove that 
	\begin{equation}\label{eq:liminf1}
		\ird  \frac{h_{\bar u}^2\bar u^2}{|x|^2} dx\le\lim_n\ird  \frac{h_{u_n}^2u_n^2}{|x|^2} dx.
	\end{equation}
By \eqref{eq:convboun} we have that $u_n \to \bar u$ in $L^2(B_{|x|})$, for all $x\in\RD$. This implies that  
	\begin{equation}\label{eq:convpoint}
		h_{u_n}(x)\to h_{\bar u}(x),\quad\hbox{ for all } x\in\RD.
	\end{equation}
By \eqref{eq:qo}, \eqref{eq:convpoint} and Fatou Lemma, we prove our claim \eqref{eq:liminf1}. 
\\
Using the weak lower semicontinuity property of the norms, inequality \eqref{eq:liminf1}, and formula \eqref{eq:str}, we obtain
	\begin{align*}
		&(3\a -2)\left(\lim_n \ird  \frac{h_{u_n}^2 u_n^2}{|x|^2} dx
-\ird  \frac{h_{\bar u}^2\bar u^2}{|x|^2} dx\right)
\\
		&\qquad \qquad\le \a \left(\liminf_n \|\n u_n\|_2^2 
- \|\n \bar u\|_2^2\right)
+(3\a -2)\left(\lim_n \ird  \frac{h_{u_n}^2 u_n^2}{|x|^2} dx
-\ird  \frac{h_{\bar u}^2\bar u^2}{|x|^2} dx\right)
\\
		&\qquad\qquad\le 			\lim_n \left(\a \|\n u_n\|_2^2
		+(3\a -2)\ird  \frac{h_{u_n}^2 u_n^2}{|x|^2} dx\right)
- \a \|\n \bar u\|_2^2-(3\a -2)\ird  \frac{h_{\bar u}^2\bar u^2}{|x|^2} dx=0.
	\end{align*}
By \eqref{eq:liminf1} we deduce that
\begin{equation*}
\lim_n\ird  \frac{h_{u_n}^2\bar u_n^2}{|x|^2} dx
=\ird  \frac{h_{\bar u}^2\bar u^2}{|x|^2} dx
\end{equation*}	
and, again by \eqref{eq:str}, $\lim_n\|\n u_n\|_2=\|\n\bar u\|_2$. Taking into account also \eqref{eq:convstr}, $E(u_n)\to E(\bar u)$ and we conclude.
\end{proof}

\section{Static solutions with finite charge}\label{se:carica}

In all this section we assume that $p>9$ and we prove that, in this case, the solution found in Theorem \ref{th:main1} belongs to $L^2(\RD)$. 

We fix a decreasing sequence $\{\e_n\}_n$ which tends to zero as $n \to +\infty$ and, for any $n\ge 1$, we define $u_n:=u_{\e_n}$, where $u_{\e_n}$ is as in Proposition \ref{pipi}. By Proposition \ref{pr:bdd} we know that $\{u_n\}_n$ is bounded and, up to a subsequence, weakly convergent in $\ch$. Finally let $u_0\in \calh$ be the solution found in Theorem \ref{th:main1}  as the weak limit of $\{u_n\}_n$ in $\ch$. 

	\begin{proof}[Proof of Theorem \ref{th:ul2}]
We need only to prove that $u_0\in L^2(\RD)$: this and the Strauss radial Lemma \cite{strauss} imply that $(\phi,A_0,A_1,A_2)$ is a positive energy static solution of \eqref{eq:e0} satisfying \eqref{eq:a012}.
\\
			By contradiction, assume that $u_0\notin L^2(\RD)$. Then there exists $R_\sigma>0$ such that $\|u_0\|^4_{L^2(B_{R_\sigma})}= \s> 16\pi^2$. \\
			Fix $\s'\in (16\pi^2,\s)$. Since $u_n\to u_0$ in $L^2_{\rm loc}(\RD)$ up to a subsequence, we can assume that there exists $n_0\in\N$ such that
				\begin{equation}\label{eq:norm}
					\s'\le\|u_n\|^4_{L^2(B_{R_\sigma})}\le \s+1, \quad\hbox{ for all }n\ge n_0.
				\end{equation} 
			By Proposition \ref{pr:radial}, there exists $\tau\in \left(\frac 2{p-1},\frac 14\right)$, $C_\tau>0$ and $R_\tau>0$ such that 
			\begin{equation*}
			(u_n(r))^{p-1}\le \frac{ C_\tau} {r^{\tau(p-1)}},
			\end{equation*}
			for $r\ge R_\tau$ and any $n\ge 1$. 
			In particular, since $\tau(p-1)>2$, taken $\delta>0$ such that $\g:=\frac {\s'}{16\pi^2}-\d>1$, there exists $R_\tau'$ such that
				\begin{equation}\label{eq:dec}
					(u_n(r))^{p-1}\le \frac \d {r^2},
				\end{equation}
			for $r\ge R'_h$ and any $n\ge 1$.
			Up to replace $R_\sigma$ with $R'_\tau$ and $\s$ with a larger number, we can assume $R_\sigma=R'_\tau$.\\
			Observe that, by \eqref{eq:norm} and \eqref{eq:dec}, we have that
				\begin{equation}\label{eq:mag}
					\frac{h_{u_n}^{2}(|x|)}{|x|^2} 
					- u_n^{p-1} - \frac {\g}{|x|^2}\ge 0,
				\end{equation}
			for $|x|\ge R_\sigma$.
		Now consider the problem
\begin{equation*}%\label{eqq}
\begin{cases}
\dis -\Delta w 
+ \frac{\g}{|x|^2} 
w =
0 &\hbox{ if } |x|>R_\sigma, 
\\[3mm]
w=u_n&\hbox{ if } |x|=R_\sigma,\\[3mm]
w\to 0 &\hbox{ as }|x| \to +\infty,
\end{cases}
\end{equation*}
which is solved by $w_n(x)=u_n(R_\sigma)R_\sigma^{\sqrt \g}|x|^{-\sqrt \g}$. Observe that 
\begin{multline}\label{eq:com}
-\Delta (u_n-w_n)+\frac \g{|x|^2}(u_n-w_n)\\= \left(-\frac{h_{u_n}^{2}(|x|)}{|x|^2} - \int_{|x|}^{+\infty}\frac {h_{u_n}(s)}{s}u_n^2(s)\,ds
+ u_n^{p-1} + \frac \g{|x|^2} -\eps_n \right)u_n 
\end{multline}
in $ H^{-1}(\RD\setminus \overline {B_{R_\sigma}})$ and, since $u_n-w_n=0$ in $\partial B_{R_\sigma}$ and $u_n - w_n\to 0$ as $|x| \to +\infty$,
we have that $(u_n-w_n)^+\in H^{1}_0(\RD\setminus \overline {B_{R_\sigma}})$.\\
So, multiplying in \eqref{eq:com} by $(u_n-w_n)^+$ and integrating, by \eqref{eq:mag} and the fact that $u_n>0$ we have
		\begin{multline*}
			\int_{|x|\ge R_\sigma}|\n (u_n - w_n)^+|^2\,dx +\int_{|x|\ge R_\sigma}\frac \g{|x|^2}( (u_n - w_n)^+)^2\,dx\\
			= \int_{|x|\ge R_\sigma}\left(-\frac{h_{u_n}^{2}(|x|)}{|x|^2} - \int_{|x|}^{+\infty}\frac {h_{u_n}(s)}{s}u_n^2(s)\,ds
			+ u_n^{p-1} + \frac \g{|x|^2}-\eps_n \right)u_n (u_n-w_n)^+\,dx\le 0
		\end{multline*}
and  then, for $|x|\ge R_\sigma$ and any $n\ge n_0$, $0\le u_n\le w_n$.
\\
%	\begin{equation*}
%		u_n(x)\le w(x)= u_n(R_\sigma)R_\sigma^{\sqrt \g}|x|^{-\sqrt \g}.
%	\end{equation*}
In conclusion, by Proposition \ref{pr:radial},
	\begin{align*}
		\|u_n\|^2_{L^2(\RD\setminus \overline {B_{R_\sigma}})}&\le 2\pi u_n^2(R_\sigma)R_\sigma^{2\sqrt \g}\int_{R_\sigma}^{+\infty}r^{1-2\sqrt \g}\,dr\\
		&\le \frac{2\pi}{2\sqrt \g - 2}u_n^2(R_\sigma)R_\sigma^{2}\\
		&\le \frac{2\pi}{2\sqrt \g - 2}\frac{ C_\tau} {R_\sigma^{2\tau}}R_\sigma^{2}
	\end{align*}
By this and \eqref{eq:norm} we deduce that $\{u_n\}_n$ is (up to a subsequence) bounded in $L^2(\RD)$ and so also in $\H$. 
Then, there exists $u\in H^1(\RD)$ and a subsequence of $\{u_n\}_n$ such that 
$u_n\rightharpoonup u$ in $H^1(\RD)$.
Since we can assume that the same subsequence is such that $u_n\to u_0$ a.e., we have $u_0=u\in L^2(\RD)$, and we obtain the contradiction.

	\end{proof}

\begin{remark}
Using similar arguments as before and taking into account the Strauss Lemma \cite{strauss}, we have that for any $\tau\in(0,a)$ there exists $C_\tau>0$ and $R_\tau>0$ such that
	\begin{equation*}
	|u_0(r)|\le \frac{C_\tau}{r^{\max (1/2,\sqrt \tau)}},\,\quad\hbox{ uniformly for } r\ge R_\tau,
	\end{equation*}  
	where $a=\lim_{r\to +\infty}h_{u_0}^2(r)$.
\end{remark}

\begin{remark}
Arguing as in the proof of Theorem \ref{th:ul2}, if $\|u_0\|_2>16 \pi^2$, then $\{u_n\}_n$ is bounded in $L^2(\RD)$.
\end{remark}

\appendix

\section{}\label{appendix}

\
\\
By H\"older inequality it is easy to see that if $u\in L^4_{\rm loc}(\RD)$ radially symmetric, 
%for some $q>4$, 
then the function 
\[
U_u(x):=
\begin{cases}
\frac{h_{u}^{2}(|x|)}{|x|^2} & \hbox{if }x\neq 0,
\\
0& \hbox{if }x= 0,
\end{cases}
\]
is well defined in $\RD$.

In the following for a measurable function $u:\RD \to \R$,  we want to understand under which assumptions on $u$ we have that
\[
V_u(x):=\int_{|x|}^{+\infty}\frac {h_u(s)}{s}u^2(s)\,ds
\]
 is well defined.

\begin{lemma}\label{app-a1}
If $u\in L^q(\RD)$ and is radially symmetric with $q\in (2,4)$, then $V_u$ is well defined in $\RD\setminus\{0\}$.
\end{lemma}

\begin{proof}
Fix $x\neq 0$.
Observe that for any $s>|x|$, by H\"older inequality we have
\begin{align*}
h_u(s)=c\int_{B_s}u^2\, dy
\le c\left(\int_{B_s}1\, dy\right)^\frac{q-2}q \left(\int_{B_s}|u|^q\, dy\right)^\frac 2q
\le c s^\frac{2q-4}q.
\end{align*}
Therefore, being $q<4$, we have
\begin{align*}
V_u(x)\le c\int_{B_{|x|}^c}\frac {u^2(y)}{|y|^\frac 4q}dy
\le c\left(\int_{B_{|x|}^c}\frac 1{|y|^\frac 4{q-2}}\, dy\right)^\frac{q-2}q \left(\int_{B_{|x|}^c}|u|^q\, dy\right)^\frac 2q<+\infty.
\end{align*}
\end{proof}

\begin{lemma}\label{app-a2}
If $u\in L^q(\RD)$ and is radially symmetric with $q\in (2,4)$ and $u\in L^\tau_{\rm loc}(\RD)$ with $\tau\in(4,+\infty)$, then $V_u$ is well defined in $\RD$.
\end{lemma}

\begin{proof}
By Lemma \ref{app-a1}, we have to prove only that $V_u(0)<+\infty$.
\\
Observe that
\begin{align*}
V_u(0)&=\int_{0}^{+\infty}\frac {h_u(s)}{s}u^2(s)\,ds
=\int_{0}^{1}\frac {h_u(s)}{s}u^2(s)\,ds
+
\int_{1}^{+\infty}\frac {h_u(s)}{s}u^2(s)\,ds
\\
&=\int_{0}^{1}\frac {h_u(s)}{s}u^2(s)\,ds
+V_u(1)=A_u+V_u(1).
\end{align*}
By Lemma \ref{app-a1}, we need to estimate only $A_u$. Since 
\begin{align*}
h_u(s)=c\int_{B_s}u^2\, dy
\le c\left(\int_{B_s}1\, dy\right)^\frac{\tau-2}\tau \left(\int_{B_s}|u|^\tau\, dy\right)^\frac 2\tau
\le c s^\frac{2\tau-4}\tau,
\end{align*}
being $\tau>4$, we have
\begin{align*}
A_u&\le c\int_{B_{1}}\frac {u^2(y)}{|y|^\frac 4\tau}dy
\le c\left(\int_{B_{1}}\frac 1{|y|^\frac 4{\tau-2}}\, dy\right)^\frac{\tau-2}\tau \left(\int_{B_{1}}|u|^\tau dy\right)^\frac 2\tau<+\infty.
\end{align*}
\end{proof}

\begin{remark}
By \cite{BHS}, we already know that, if $u\in L^2(\RD)\cap  L^\infty_
{\rm loc}(\RD)$, then $V_u\in L^\infty(\RD)$.
\end{remark}


\begin{thebibliography}{99}

\bibitem{APP}
A. Azzollini, L. Pisani and A. Pomponio, {\em Improved estimates and a limit case for the electrostatic Klein-Gordon-Maxwell system}, Proc. Roy. Soc. Edinburgh Sect. A {\bf 141}, (2011), 449--463.

\bibitem{BBS}
J. Bellazzini, C. Bonanno and G. Siciliano,
{\em Magneto-Static Vortices in Two Dimensional Abelian Gauge Theories}, 
Mediterr. J. Math. {\bf 6}, (2009), 347--366.


\bibitem{BL}
H. Berestycki and P.L. Lions,  {\em Nonlinear scalar field equations. I. Existence of a ground state}, Arch. Rational Mech. Anal. {\bf 82} (1983), 313--345.

\bibitem{berge}
L.~Berg\'{e}, A.~de Bouard and J.C.~Saut, 
{\em Blowing up time-dependent solutions 
of the planar Chern-Simons gauged nonlinear Schr\"{o}dinger equation}, 
Nonlinearity {\bf 8} (1995), 235-253.

\bibitem{B}
H.~Brezis, 
{\it Functional Analysis, Sobolev Spaces and Partial Differential Equations},
Springer, New York, 2011.


\bibitem{BHS}
J.~Byeon, H.~Huh and J.~Seok,
{\em Standing waves of nonlinear Schr\"{o}dinger equations with the
gauge field},
J. Funct. Anal. {\bf 263} (2012), 1575--1608.

\bibitem{BHS2}
J.~Byeon, H.~Huh and J.~Seok, 
{\em On standing waves with a vortex point of order $N$ for the nonlinear 
Chern- Simons-Schr\"odinger equations}, 
J. Diff. Eqns. {\bf 261} (2016), 1285--1316.


\bibitem{cunha}
P.L.~Cunha, P.~d'Avenia, A.~Pomponio and G.~Siciliano,
{\em A multiplicity result for Chern-Simons-Schr\"odinger equation with a general nonlinearity},
Nonlinear Differ. Equ. Appl. {\bf 22} (2015), 1831--1850.
%
%\bibitem{Fel}
%B.~Felsager,
%{\em Geometry, particles and fields}, (1998), Springer-Verlag.

\bibitem{GT}
D. Gilbarg and N. Trudinger, Elliptic Partial Differential Equations of Second Order, second edition, Grundlehren Math. Wiss., vol. 224, Springer, Berlin, 1983.

\bibitem{hagen}
C. Hagen, {\em A new gauge theory without an
elementary photon}, Ann. of Phys. {\bf 157} (1984), 342--359.

\bibitem{hagen2}
C. Hagen, {\em Rotational anomalies without anyons}, Phys.
Review D {\bf 31} (1985), 2135--2136.



\bibitem{huh}
H. Huh, {\em Blow-up solutions of the Chern-Simons-Schr\"{o}dinger equations}, 
Nonlinearity {\bf 22}  (2009),  967--974.

\bibitem{huh2}
H.~Huh, 
{\em Standing waves of the Schr\"{o}dinger equation coupled with the Chern-Simons gauge field},
J. Math. Phys. {\bf 53} (2012), 063702.

\bibitem{huh3}
H.~Huh, 
{\em Energy Solution to the Chern-Simons-Schr\"{o}dinger equations},
Abstr. Appl. Anal. {\bf 2013}, Article ID 590653, 7 pp.

\bibitem{jackiw1}
R. Jackiw, 
{\em Invariance, symmetry and periodicity in gauge theories}, 
Acta Phys. Austriaca, Suppl. XXII, {\bf 383} (1980).

\bibitem{jackiw0}
R.~Jackiw and S.Y.~Pi,
{\em Soliton solutions to the gauged nonlinear Schr\"{o}dinger equations on the plane}, 
Phys. Rev. Lett. {\bf 64} (1990), 2969--2972.


\bibitem{jackiw}
R. Jackiw and S.Y. Pi, 
{\em Classical and quantal nonrelativistic Chern-Simons theory}, 
Phys. Rev. D {\bf 42} (1990), 3500--3513.


\bibitem{jackiw2}
R.~Jackiw and S.Y.~Pi,
{\em Self-dual Chern-Simons solitons}, 
Progr. Theoret. Phys. Suppl. {\bf 107} (1992), 1--40.

\bibitem{JPR}
Y.~Jiang, A.~Pomponio and D.~Ruiz, 
{\em Standing waves for a gauged nonlinear Schr\"odinger equation with a vortex point}, 
Commun. Contemp. Math. {\bf 18} (2016), no.4, 1550074, 20 pp.


%\bibitem{LU}
%O.~A.~Ladyzhenskaya and N.~N.~Uraltseva, 
%{\em Linear and quasilinear elliptic equations},
%Academic Press, New York, 1968.

\bibitem{LL}
E.H. Lieb and M. Loss, Analysis. Second edition. Graduate Studies in Mathematics, 14. American Mathematical Society, Providence, RI, 2001. xxii+346 pp.


\bibitem{L1}
P.L. Lions,
\textit{The concentration-compactness principle in the calculus of variation.
The locally compact case. Part I},
Ann. Inst. Henri Poincar\'e, Anal. Non Lin\'eaire, {\bf 1}, (1984), 109--145.


\bibitem{L2}
P.L. Lions,
\textit{The concentration-compactness principle in the calculus of variation.
The locally compact case. Part II},
Ann. Inst. Henri Poincar\'e, Anal. Non Lin\'eaire, {\bf 1}, (1984), 223--283.


\bibitem{liu}
B.~Liu and P.~Smith, 
{\em Global wellposedness of the equivariant Chern-Simons-Schr\"{o}dinger equation}, 
Rev. Mat. Iberoam. {\bf 32} (2016), 751--794.

\bibitem{tataru}
B.~Liu, P.~Smith and D.~Tataru, 
{\em Local wellposedness of Chern-Simons-Schr\"{o}dinger}, 
Int. Math. Res. Not. IMRN {\bf 23} (2014), 6341--6398.

\bibitem{oh}
S.J. Oh and F. Pusateri, {\em Decay and Scattering for the Chern-Simons-Schr\"{o}dinger Equations}, 
Int. Math. Res. Not. IMRN {\bf 24} (2015), 13122--13147. 

\bibitem{P}
A.~Pomponio, 
{\em Some results on the Chern-Simons-Schr\"odinger equation}, 
Lect. Notes Semin. Interdiscip. Mat. {\bf 13} (2016), 67--93.

\bibitem{AD}
A.~Pomponio and D.~Ruiz, 
{\em A variational analysis of a gauged nonlinear Schr\"{o}dinger equation}, 
J. Eur. Math. Soc. {\bf 17} (2015), 1463--1486.

\bibitem{AD2}
A.~Pomponio and D.~Ruiz, 
{\em Boundary concentration of a gauged nonlinear Schr\"odinger equation}, 
Calc. Var. PDE {\bf 53} (2015), 289--316.

\bibitem{R}
D. Ruiz, {\em On the Schrödinger-Poisson-Slater system: Behavior of minimizers, radial and nonradial cases}, Arch. Ration. Mech. Anal. {\bf 198} (2010), 349--368.

\bibitem{strauss}
W.A.~Strauss,
{\em Existence of solitary waves in higher dimensions}, 
Comm. Math. Phys. {\bf 55} (1977), 149--162.
%
%\bibitem{struwe}
%M. Struwe, {\em On the evolution of harmonic mappings of Riemannian surfaces}, Comment. Math. Helv. {\bf 60} (1985), 558--581.

\bibitem{WT}
Y.~Wan and J.~Tan,
{\em The existence of nontrivial solutions to Chern-Simons-Schr\"odinger systems},
Disc. Cont. Dyn. Syst. {\bf 37} (2017), 2765--2786.

\bibitem{Y}
J.~Yuan,
{\em Multiple normalized solutions of Chern-Simons-Schr\"odinger system},
Nonlinear Differ. Equ. Appl. {\bf 22} (2015), 1801--1816.

\end{thebibliography}
\end{document}